\documentclass[10pt,a4paper]{article}

\usepackage[utf8]{inputenc}
\usepackage{amsmath}
\usepackage{amsfonts}
\usepackage{amssymb}
\usepackage{amsthm}
\usepackage[headsep=10mm]{geometry}
\usepackage{caption}
\usepackage{subcaption}
\usepackage[pdftex]{graphicx}
\usepackage{bm}
\usepackage{fancyhdr}
\usepackage{dsfont}
\usepackage{stmaryrd}
\usepackage{enumitem}
\usepackage{color}

\newtheorem{theorem}{Theorem}
\newtheorem{lemma}{Lemma}
\newtheorem{definition}{Definition}
\newtheorem{assumption}{Assumption}

\numberwithin{equation}{section}

\newcommand{\grad}[1]{{\bm{\nabla}{#1}}}
\newcommand{\dv}[1]{\bm{\nabla\cdot{#1}}}

\newcommand{\pfrac}[2]{\frac{\partial{#1}}{\partial{#2}}}
\newcommand{\dx}[1]{\hspace{0.75mm}\mathrm{d}{#1}}

\newcommand{\curlybrackets}[1]{\left\lbrace{#1}\right\rbrace}
\newcommand{\newR}[1]{\textcolor{black}{{#1}}}
\newcommand{\newB}[1]{\textcolor{black}{{#1}}}
\newcommand{\newP}[1]{\textcolor{black}{{#1}}}

\begin{document}
\title{The `Recovered Space' Advection Scheme for Lowest-Order Compatible Finite Element Methods}
\author{Thomas M. Bendall, Colin J. Cotter, Jemma Shipton}
\maketitle

\begin{abstract}
\noindent
We present a new compatible finite element advection scheme for the compressible Euler equations.
Unlike the discretisations described in Cotter and Kuzmin (2016) and Shipton et al (2018), 
the discretisation uses the lowest-order family of compatible finite element spaces, 
but still retains second-order numerical accuracy.
This scheme obtains this second-order accuracy by first `recovering' the function in higher-order spaces,
before using the discontinuous Galerkin advection schemes of Cotter and Kuzmin (2016).
As well as describing the scheme, we also present its stability properties and a strategy for ensuring boundedness.
We then demonstrate its properties through some numerical tests, before presenting its use within a model solving the compressible Euler equations.
\end{abstract}

\noindent \small{\textit{Keywords}:
Advection scheme;
Discontinuous Galerkin;
Compatible finite element methods;
Numerical weather prediction} 
\section{Introduction}
\label{sec:intro}
Over the past few decades, technological improvements in parallel computing 
have driven increased performance of numerical models of weather and 
climate systems, allowing them to be run at increasingly finer resolutions
and delivering significantly improved predictive capabilities.
However, the traditional latitude-longitude grids used in these models
are leading to the approach of a scalability bottleneck:
improvements in potential computational power no longer lead to
improved model performance, as the clustering of points around the poles
limits the rate of data communication (the so-called `pole problem').
This has led to a search for alternative grids for these models. \\
\\
However, latitude-longitude grids (coupled with Arakawa C-grid staggering)
provided many properties important for accurate representation of
atmospheric motion: avoiding spurious pressure modes, accurately 
representing geostrophic balance and not supporting spurious gravity-inertia
or Rossby wave modes.
These properties are described in \cite{staniforth2012horizontal}.
It is therefore desirable that any alternative grid should 
maintain these properties. \\
\\
Finite element methods have therefore been explored, as they offer
the opportunity to solve the equations of motion with arbitrary
mesh structures.
Extending previous work into the use of mixed finite element methods
for such geophysical fluid models, \cite{cotter2012mixed} proposed
mixed finite element methods with two crucial properties.
Firstly, that the finite element spaces used should be \textit{compatible}
with one another, so that the discrete versions of the vector calculus
operators preserve the properties that $\mathrm{div}(\mathrm{curl})=0$
and that $\mathrm{curl}(\mathrm{grad})=0$. \\
\\
Secondly, the function spaces used for the velocity and pressure variables
should be chosen to fix the ratio of velocity to pressure degrees of freedom
per element to be 2:1.
This was shown in \cite{cotter2012mixed} to be a necessary condition for the
avoidance of spurious inertia-gravity and Rossby wave modes.
\cite{cotter2012mixed} also proposed two sets of finite element spaces
that would meet these criteria on quadrilateral and triangular elements. \\
\\
\cite{cotter2012mixed} has been followed by a series of works into the use
of compatible finite element methods for use in numerical weather models,
for instance \cite{staniforth2013analysis}, \cite{cotter2014finite}, \cite{mcrae2014energy},
\cite{natale2016compatible} and \cite{shipton2018higher}.
Most relevant to the work here are \cite{yamazaki2017vertical}, which presented a full discretisation
of the Euler-Boussinesq equations in the context of compatible finite element methods;
and \cite{cotter2016embedded}, which introduced a 
embedded discontinuous Galerkin transport scheme. \\
\\
The compatible finite element framework thus suggests using particular
compatible finite element spaces for the velocity $\bm{v}$ and density $\rho$.
For horizontal grids, the examples given in \cite{cotter2012mixed} are
$\bm{v}\in \mathrm{RT}_{k+1}$ with $\rho\in \mathrm{Q}_k$ on quadrilateral elements,
and $\bm{v}\in \mathrm{BDFM}_k$ with $\rho\in \mathrm{P}_k$ on triangular elements.
Here $\mathrm{RT}_{k}$ is the space of $k$-th degree Raviart-Thomas quadrilateral elements,
$Q_k$ is the space of discontinuous $k$-th order polynomial elements on quadrilaterals,
$\mathrm{BDFM}_k$ is the space of $k$-th degree Brezzi-Douglas-Fortin-Marini triangular elements
and finally $\mathrm{P}_k$ is the space of $k$-th order polynomial elements on triangles.
Most of these elements appear in the Periodic Table of Finite Elements \cite{arnold2014periodic}. \\
Another set of spaces to consider are those used in 2D vertical slices, which are relevant for test cases
used in developing the model.
These models are constructed as the tensor product $U\times V$ of a 1D horizontal space $U$ with a
1D vertical space $V$.
We will denote the space of $k$-th degree continuous polynomial elements by $\mathrm{CG}_k$,
and that of $l$-th degree discontinuous polynomials by $\mathrm{DG}_l$.
In this case horizontal velocities lie in the space $\mathrm{CG}_{k+1}\times \mathrm{DG}_{l}$,
vertical velocities are in $\mathrm{DG}_k\times \mathrm{CG}_{l+1}$, whilst the density is in
$\mathrm{DG}_k\times \mathrm{DG}_l$.
In order to mimic the Charney-Philips grid used in 
finite difference models, the potential temperature $\theta$
lies in the partially-continuous space,
$\mathrm{DG}_k \times \mathrm{CG}_{l+1}$, 
i.e. temperature degrees of freedom (DOFs) are co-located with
those for vertical velocity, as in \cite{natale2016compatible} 
and \cite{yamazaki2017vertical}.
The construction of tensor product finite element spaces is described by \cite{mcrae2016automated}.
Throughout the rest of the paper we will consider the case that $k=l$ for quadrilateral elements in this vertical slice set-up, although the advection scheme can be applied more generally. \\
\\
Advection schemes for the $k=1$ family of spaces have been presented in \cite{cotter2016embedded},
\cite{yamazaki2017vertical} and \cite{shipton2018higher}.
An advantage of the $k=1$ degree was that it is easy to
formulate for it advection schemes
with the property of
second-order numerical accuracy, i.e. that the error
associated with the discretisation of the advection
process is proportional to
$(\Delta x)^2$, the grid size squared.
This is one of the crucial properties of discretisations for 
numerical weather models
listed in \cite{staniforth2012horizontal}.
However, in the $k=1$ families, coupling of the dynamics to the
sub-grid physical processes
(for example the effects of moisture or radiation)
may be more challenging than for $k=0$ case.
The effects of such physical processes are commonly expressed
as tendencies to the prognostic variables and typically
calculated pointwise at the degrees of freedom.
For the fields that are piecewise constant or linear,
the pointwise values can be interpreted as mean quantities
for that element and the tendencies can be formulated as such.
In the $k=1$ case on quadrilateral elements, the temperature and moisture fields
are piecewise quadratic functions in the vertical,
and the physical interpretation of the values at the degrees
of freedom becomes less clear.
It is therefore desirable to consider $k=0$ spaces using
advection schemes that have second-order numerical accuracy.
The main result of this paper is thus a presentation of such an
advection scheme for this set of spaces. \\
\\
This scheme has been inspired by the results of 
\cite{georgoulis2018recovered}, which implies
that it is possible to reconstruct a discontinuous zeroth-order 
field in a continuous first-order space, via
an averaging operation that has second-order numerical
accuracy.\\
\\
After describing the scheme in Section \ref{sec:scheme}, this paper
presents several of its properties in Section \ref{sec:properties},
including a general argument of its stability and von Neumann
analysis of the scheme in three particular cases.
Section \ref{sec:numerics} presents the results of numerical 
tests demonstrating the second-order numerical accuracy of the 
scheme, the stability calculations of Section 
\ref{sec: von neumann} and the use of a limiter within
the advection scheme.
Finally, the use of the advection scheme within a model
of the compressible Euler equations is presented in Section 
\ref{sec: compressible model}.

\section{The `Recovered Space' Scheme}
\label{sec:scheme}
The key idea upon which this scheme is based is the family of
recovered finite element methods introduced by
\cite{georgoulis2018recovered}.  These methods combine features of
discontinuous Galerkin approaches with conforming finite element
methods.  \newR{They are similar to other recovery methods, such as those in
\cite{titarev2002ader,van2005discontinuous}, in that they reconstruct
higher-order polynomials from lower order data in a patch of cells.
They differ in that they do not attempt to reproduce polynomials of a
certain degree exactly.  Instead, they involve mapping discontinuous
finite element spaces to continuous ones, via recovery operators,
relying on analysis estimates of stability and accuracy.}
The scheme that we will introduce involves the use of one
of these operators to recover a function in a
discontinuous first-order space
from one in a discontinuous zeroth-order space.
To do this, we first recover a first-order continuous function 
from the zeroth-order discontinuous function using an averaging
operator described in \cite{georgoulis2018recovered} and
\cite{karakashian2007convergence}.
This operator finds the values for any degree of freedom
shared between elements in a continuous function space,
by averaging between the values of the surrounding degrees of
freedom from the discontinuous space. \\
\\
\newR{Once this operator has been applied, existing transport schemes
can be used to perform the advection upon the recovered field.
This approach is compatible when the transport equation is in
`advective' form
\begin{equation}
\pfrac{q}{t} + \bm{v}\bm{\cdot} \grad{}q = 0, \label{eqn:advective}
\end{equation}
or `conservative' form
\begin{equation}
\pfrac{q}{t} + \grad{}\bm{\cdot}(q\bm{v}) = 0, \label{eqn:conservative form}
\end{equation}
where $q$ is the quantity to be transported by velocity $\bm{v}$.
However most of our analysis will focus on the application of this scheme to
the `advective' form of the equation, under which the mass $\int_\Omega q \dx{x}$ 
over the whole domain $\Omega$ will only be necessarily conserved when
the flow is incompressible, $\dv{v}=0$.}

\subsection{The Scheme} 
First we will define a set of spaces that our functions will lie in. 
Let $V_0(\Omega)$ be the lowest-order finite element space in which the 
initial field lies, where $\Omega$ is our spatial domain\newR{\footnote{\newR{This spatial domain can be arbitrary, but with geophysical applications in mind we anticipate the scheme being used upon rectangular or cuboid domains (with or without periodicity) or spherical shells.
However the recovery operator that we consider in Section \ref{sec:recovery operator} is intended for use in flat spaces or with only scalar fields in curved spaces and we do not yet consider the application to transport of vector fields in curved spaces. 
Therefore in this work we will predominantly consider rectangular domains with a vertical coordinate, with rigid walls at the top and bottom edges.}}}.
$V_1(\Omega)$ is then the space of next degree, which will be fully discontinuous.
We also have that $V_0\subset V_1$.
$\tilde{V}_1(\Omega)$ is the fully continuous space of 
same degree as $V_1(\Omega)$,
whilst $\hat{V}_0(\Omega)$ is a 
broken (i.e. fully discontinuous) version of $V_0(\Omega)$.
In many cases, $\hat{V}_0(\Omega)$ and $V_0(\Omega)$ will coincide.\\
\\
We now define a series of operators to map between these spaces.

\begin{definition}
The recovery operator $\mathcal{R}$ acts upon a function in the 
initial space to make a function in the
continuous space of higher-degree, so that $\mathcal{R}:V_0\to\tilde{V}_1$.
The operator has second-order numerical accuracy.
\end{definition}

\begin{assumption} \label{assumption}
The recovery operator $\mathcal{R}$ has the property that for all $\rho_0\in V_0$,
there is some $C>0$ such that $||\mathcal{R}\rho_0||_{L^2} \leq C||\rho_0||_{L^2}$.
\end{assumption}

\begin{definition}
The injection operator $\mathcal{I}:V\to V_1$ identifies a function in $V_0$, $\hat{V}_0$ or $\tilde{V}_1$ as a 
member of $V_1$.
This must be numerically implemented, although it does nothing else mathematically.
\end{definition}

\begin{definition}
The projection operator
$\hat{\mathcal{P}}:\tilde{V}_1\to\hat{V}_0$, is defined to give 
$\hat{u}=\hat{\mathcal{P}}\tilde{v}$, from $\tilde{v}\in\tilde{V}_1$,
by finding the solution $\hat{u}\in \hat{V}_0$ to
\begin{equation}
\int_\Omega \hat{\psi} \hat{u} \dx{x} = 
\int_\Omega \hat{\psi} \tilde{v} \dx{x} ,  
\ \ \ \ \ \forall \hat{\psi}\in \hat{V}_0.
\end{equation}
\end{definition}

\begin{definition}
The advection operator $\mathcal{A}:V_1\to V_1$, \newR{represents the action
of performing one time step of a stable discretisation of the advection equation
(in either advective or conservative form) and has second-order numerical accuracy in space.}
\end{definition}

\begin{definition}
The projection operator $\mathcal{P}:V_1\to V_0$
will have two forms. 
The first, $\mathcal{P}_A$, is defined to give $u=\mathcal{P}_Av$
from $v\in V_1$, by finding the solution $u\in V_0$ to
\begin{equation}\label{eqn:projectionA}
\int_\Omega \psi u \dx{x} = 
\int_\Omega \psi v \dx{x} ,  
\ \ \ \ \ \forall \psi\in V_0,
\end{equation}
where $u\in V_0$ and $v\in V_1$. \\
\\
The second form, $\mathcal{P}_B$, is
composed of two operations: $\mathcal{P}_I:V_1\to\hat{V}_0$, interpolation
into the broken space by pointwise evaluation at degrees of freedom, and $\mathcal{P}_R:\hat{V}_0\to V_0$, recovery
from the broken space to the original space, restoring continuity
via the reconstruction operator from \cite{georgoulis2018recovered}.
$\mathcal{P}_B$ can thus be written as $\mathcal{P}_B = 
\mathcal{P}_R\mathcal{P}_I$.
\end{definition}
\noindent
In the case that $V_0$ is fully discontinuous, $\mathcal{P}_A$ and $\mathcal{P}_B$
will be identical operations.
However for fully or partially continuous $V_0$,
$\mathcal{P}_B$ prevents the formation of any new maxima and minima,
whereas $\mathcal{P}_A$ does not.
We may thus use $\mathcal{P}_B$ as the projection operator when trying to bound
the transport, such as for a moisture species.
Further discussion can be found in Section \ref{sec:limiting}.
The drawback is that whilst $\mathcal{P}_A$ preserves the mass
(setting $\psi=1$ gives $\int_{\Omega}u\dx{x}=\int_\Omega v\dx{x}$),
$\mathcal{P}_B$ does not necessarily do so.

\begin{definition}
The `recovered space' scheme then takes the function
$\rho_0^n\in V_0$ at the $n$-th time step and returns the function 
$\rho_0^{n+1}\in V_0 $ at the $(n+1)$-th time step by performing
the following series of operations:
\begin{equation}
\rho_0^{n+1} = \mathcal{PAI}
(\mathcal{R}-\hat{\mathcal{P}}\mathcal{R}+1)\rho_0^n,
\end{equation}
where $\mathcal{P}$ could be either $\mathcal{P}_A$ or $\mathcal{P}_B$.
\end{definition}
\noindent
An important property of this scheme is that in the absence of flow, the field being advected must remain unchanged.
In this case $\mathcal{A}$ will be the identity, and since 
$\mathcal{PIR} \equiv \mathcal{PI}\hat{\mathcal{P}}\mathcal{R}$, then
$\rho_0^{n+1}=\mathcal{PI}\rho^n_0 = \rho_0^n$.
\newB{In practice, mass will be only be conserved up to the precision used by the numerical solver for $\mathcal{P}$.}

\subsection{Example Spaces}\label{sec: example spaces}
In this section we give an example set of spaces 
$\lbrace V_0, V_1, \tilde{V}_1, \hat{V}_0\rbrace$ on quadrilateral elements
that can be used for this scheme,
in the context of 2D vertical slice problems.
The variables that we will consider are the density $\rho$, velocity $\bm{v}$ and
potential temperature $\theta$. \\
\begin{table}[h!]
\centering
\begin{tabular}{c|c|c|c|c}
Variable & $V_0$ &  $V_1$ & $\tilde{V}_1$ & $\hat{V}_0$ \\
\hline
& & & & \\
$\rho$ & $\mathrm{DG}_0\times \mathrm{DG}_0$ 
& 
$\mathrm{DG}_1\times \mathrm{DG}_1$ & $\mathrm{CG}_1\times \mathrm{CG}_1$ &
$\mathrm{DG}_0\times \mathrm{DG}_0$ \\
& \includegraphics[scale=0.02]{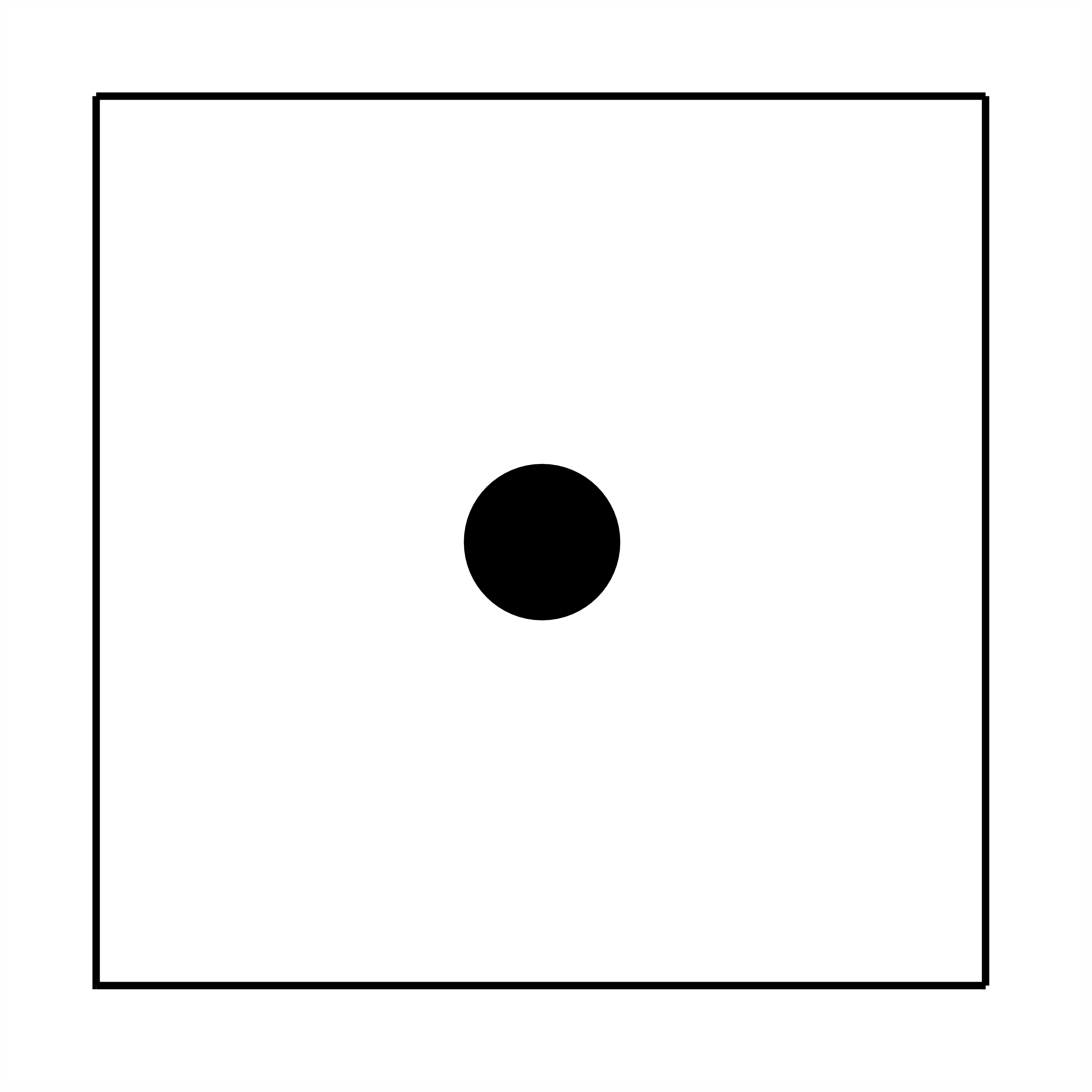}
& \includegraphics[scale=0.02]{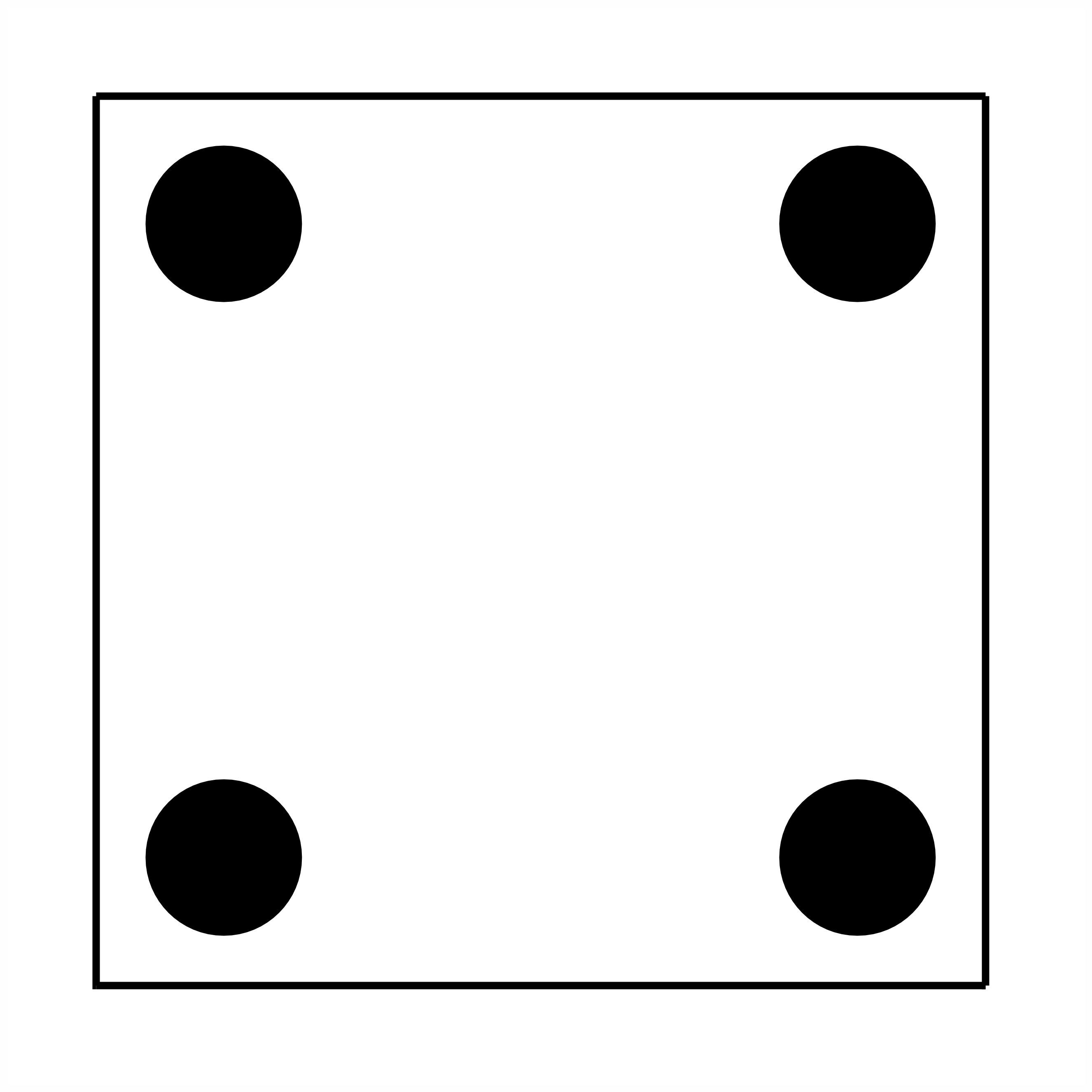}
& \includegraphics[scale=0.02]{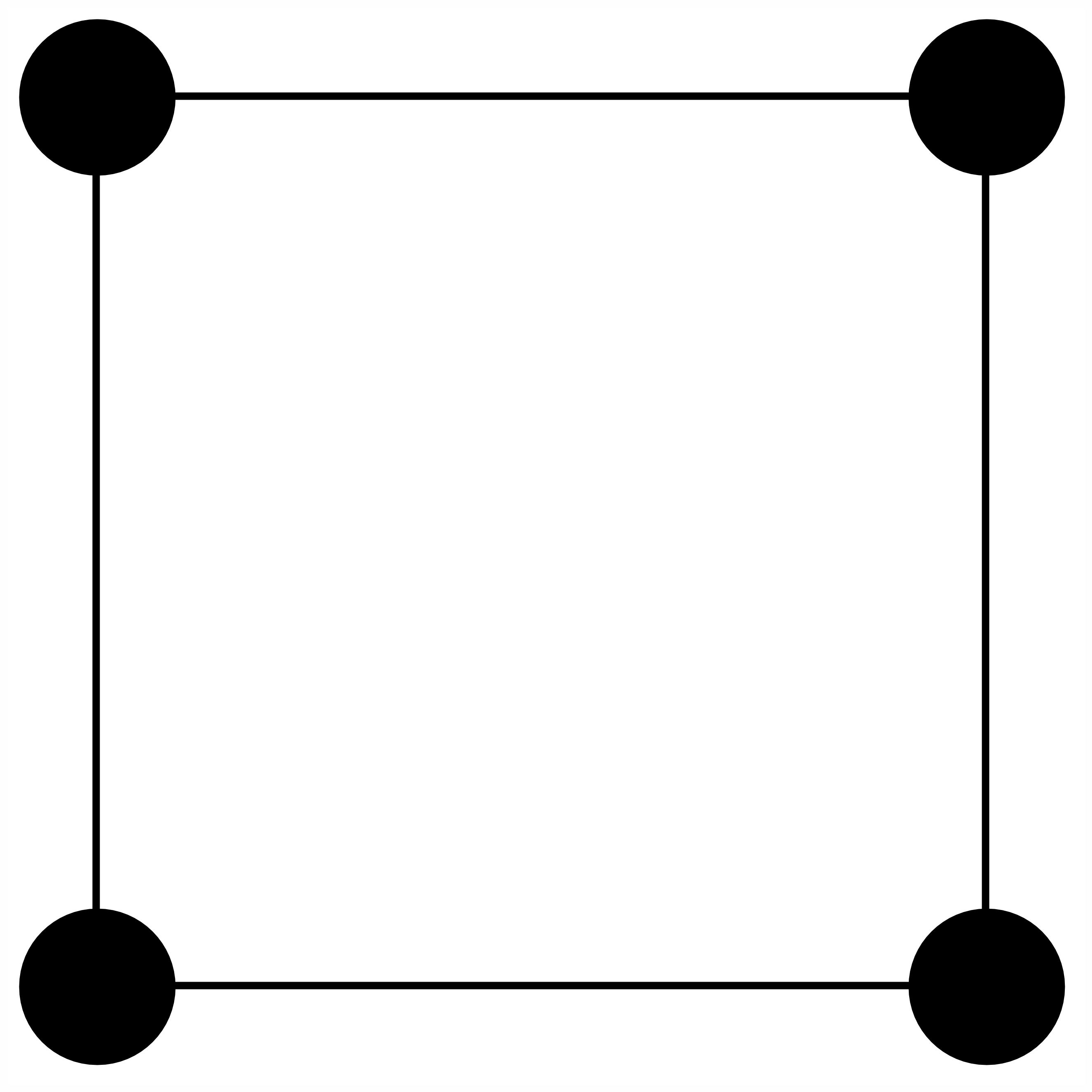}
& \includegraphics[scale=0.02]{DG0.png}
\\
\hline
& & & & \\
$\bm{v}$ & $\mathrm{RT}_0$ & 
Vector $\mathrm{DG}_1\times \mathrm{DG}_1$ & 
Vector $\mathrm{CG}_1\times \mathrm{CG}_1$ &
Broken $\mathrm{RT}_0$ \\
& \includegraphics[scale=0.02]{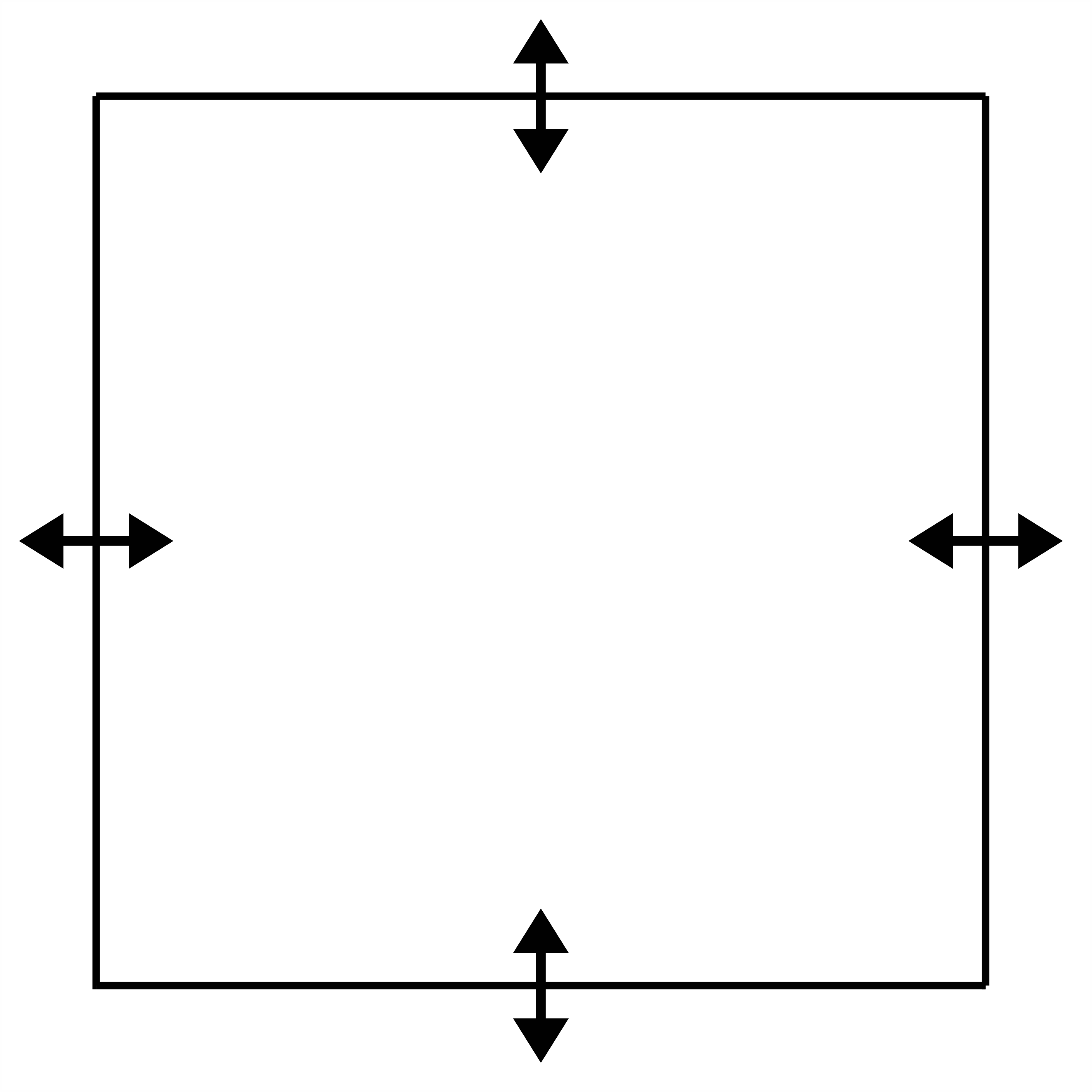}
& \includegraphics[scale=0.02]{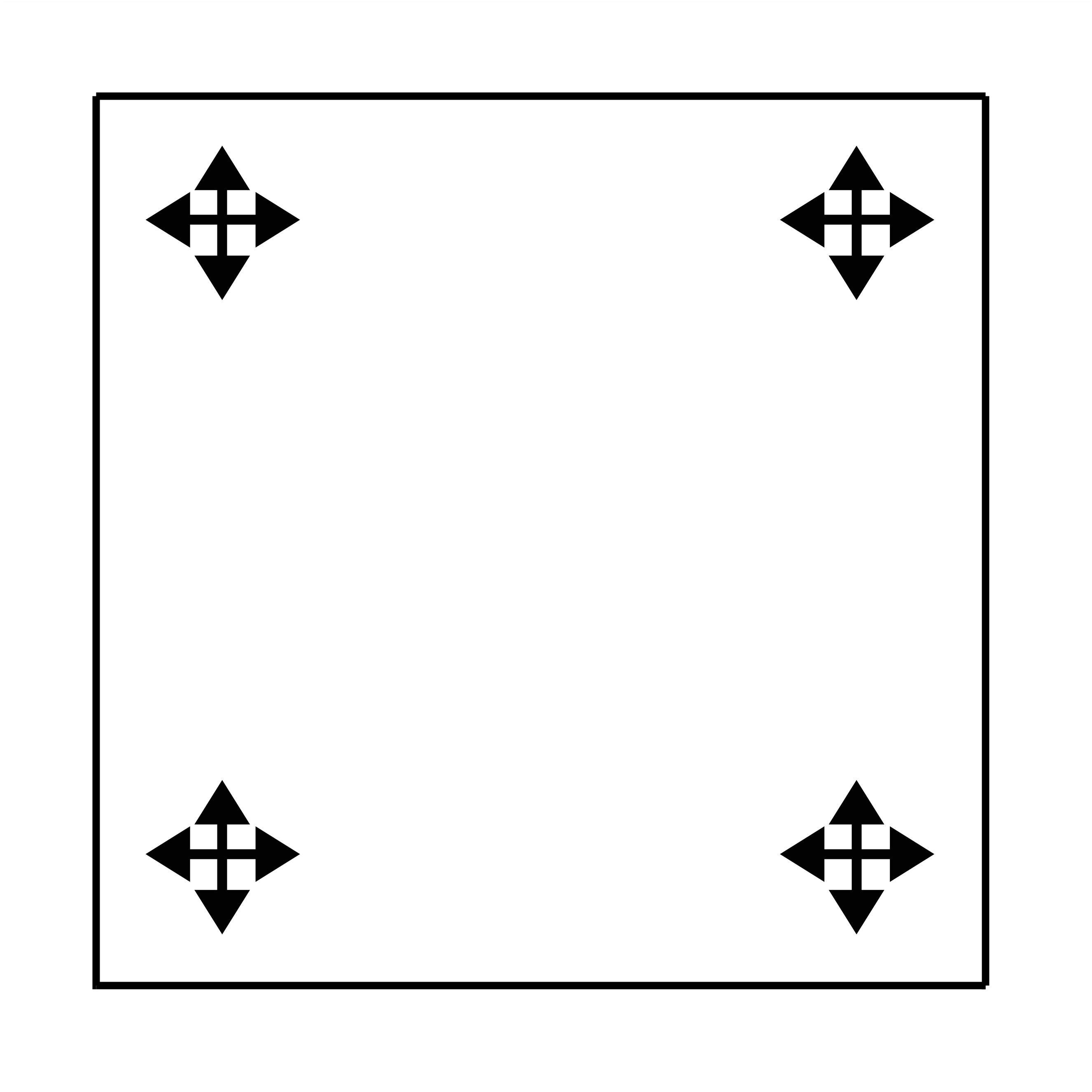}
& \includegraphics[scale=0.02]{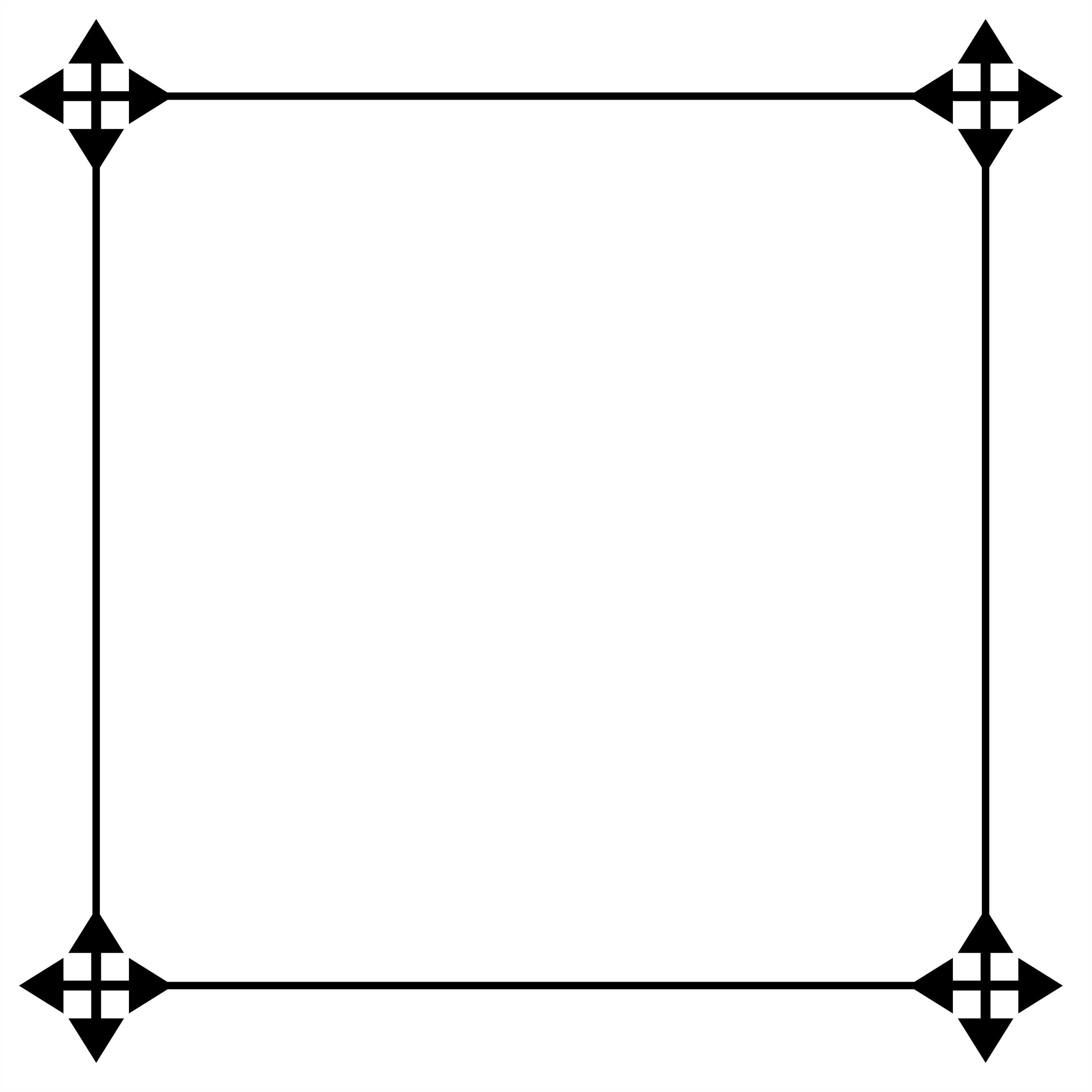}
& \includegraphics[scale=0.02]{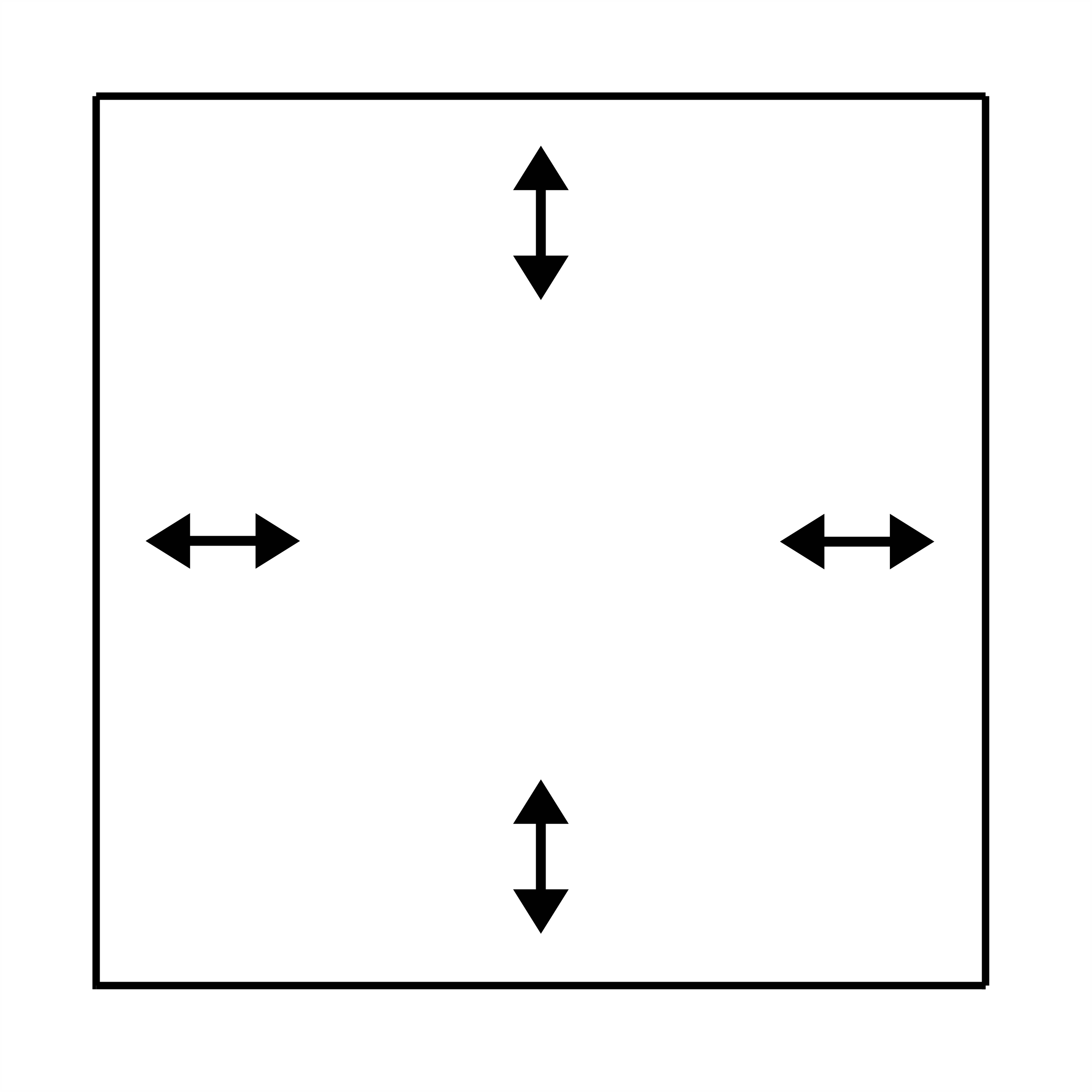}
\\
\hline
& & & & \\
$\theta$ & $\mathrm{DG}_0\times \mathrm{CG}_1$ & 
$\mathrm{DG}_1\times \mathrm{DG}_1$ & $\mathrm{CG}_1\times \mathrm{CG}_1$ &
$\mathrm{DG}_0\times \mathrm{DG}_1$ \\
& \includegraphics[scale=0.02]{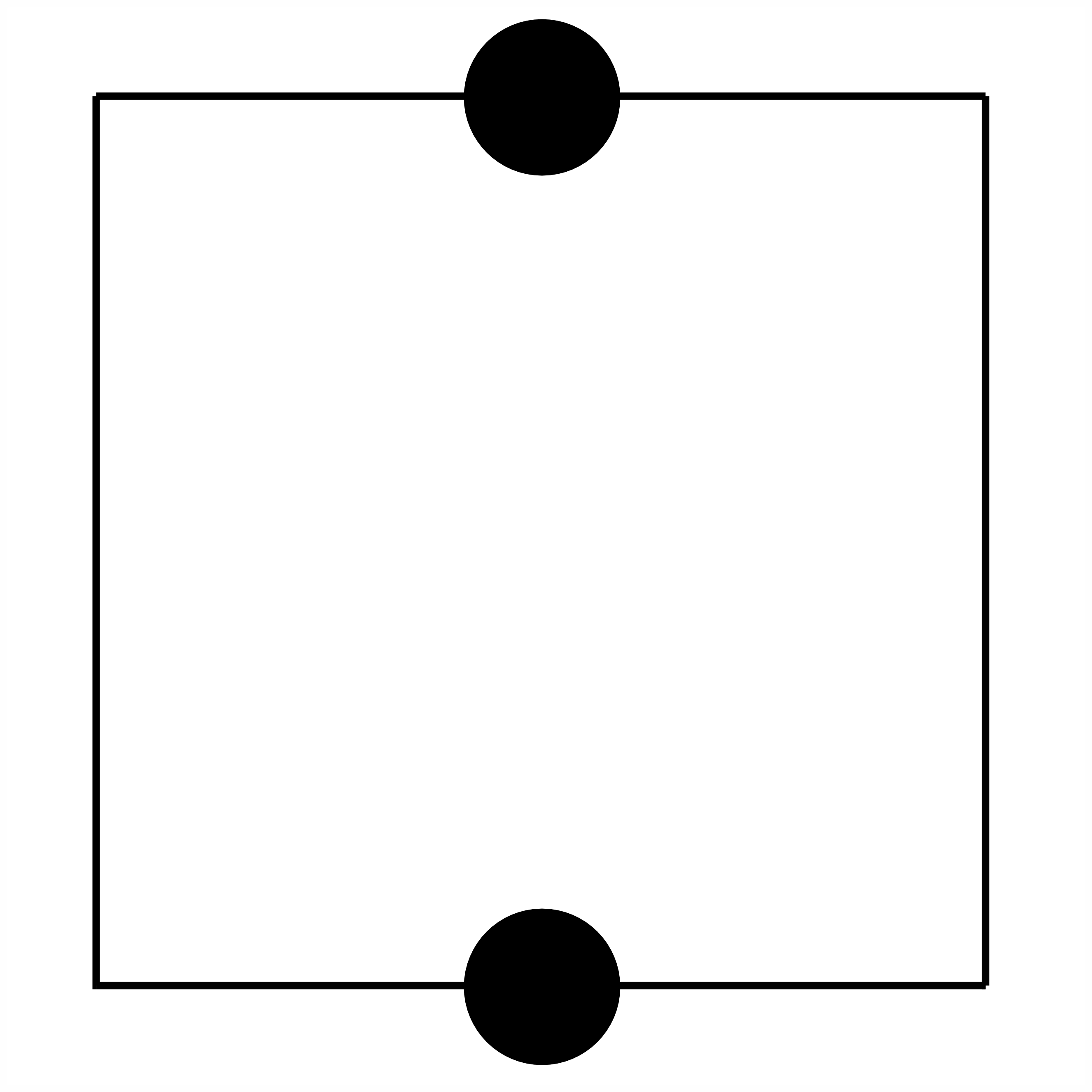}
& \includegraphics[scale=0.02]{DG1.png}
& \includegraphics[scale=0.02]{CG1.png}
& \includegraphics[scale=0.02]{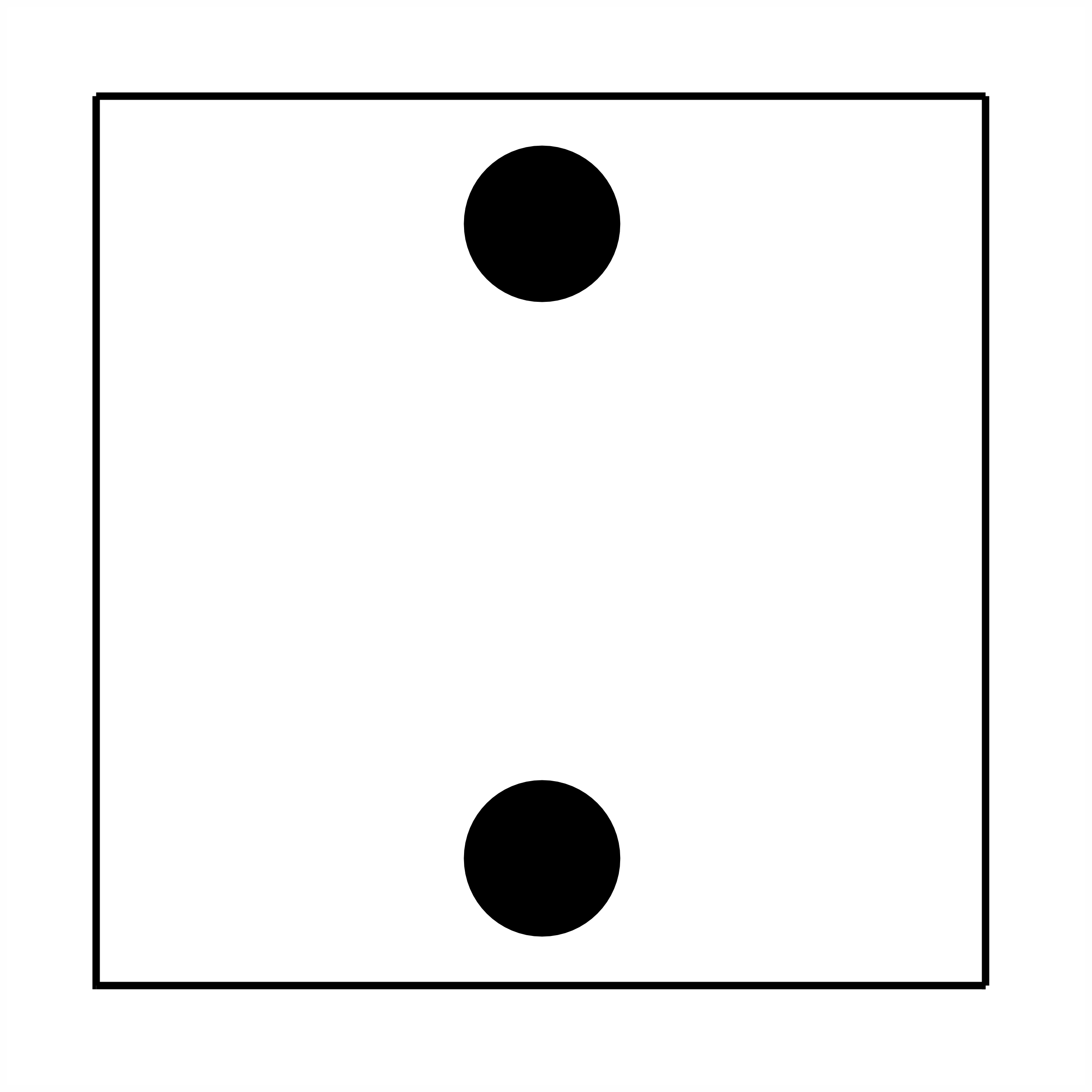}
\end{tabular}
\caption*{\small{\textbf{Table 1}:
An example set of spaces $\lbrace V_0, V_1, \tilde{V}_1, \hat{V}_0 \rbrace$
to be used in the recovered advection scheme, for a vertical slice model on
quadrilateral elements.
These spaces are given for the density $\rho$, velocity $\bm{v}$ and potential
temperature $\theta$.
$V_0$ represents the original space of the function whilst $V_1$ is the space in which
the advection takes place.
$\tilde{V}_1$ the fully-continuous space of same degree as $V_1$ 
and $\hat{V}_0$ is the fully-discontinuous version of $V_0$.
The diagrams represent the locations of the degrees of freedom for each of these spaces.
The space $\mathrm{RT}_k$ is the $k$-th Raviart-Thomas space, while $V_h\times V_v$
represents the tensor product of the horizontal space $V_h$ with the vertical space $V_v$.}}
\end{table} \\
For each variable the space $V_0$ is the normal space in which the variable lies.
Since we are motivated by using the lowest-order family
of compatible finite element spaces on quadrilateral elements,
that is what we will use here.
For $\bm{v}$ this is $V_0=\mathrm{RT}_0$, 
the lowest-order Raviart-Thomas space with vector DOFs
that have normal components continuous over cell boundaries.
The density $\rho$ lies $\mathrm{DG}_0\times \mathrm{DG}_0$, 
which has a single DOF at the centre
of the cell, and $\theta\in \mathrm{DG}_0\times\mathrm{CG}_1$,
i.e. discontinuous constant values in the horizontal but continuous linear in the vertical.
The $\theta$ DOFs are co-located with those for vertical velocity. \\
\\
For the advection operator $\mathcal{A}$ to have second-order numerical accuracy,
the advection should take place in spaces that are at least linear in each direction.
We therefore choose $V_1$ to be the smallest entirely discontinuous space that is linear
in both directions.
For $\rho$ and $\theta$, this is $\mathrm{DG}_1\times\mathrm{DG}_1$,
whilst for $\bm{v}$ this is the vector $\mathrm{DG}_1\times\mathrm{DG}_1$ space,
(and so is $\mathrm{DG}_1\times\mathrm{DG}_1$ for each component).\\
\\
The space $\tilde{V}_1$ is then formed by taking the completely continuous
form of $V_1$, whilst $\hat{V}_0$ is formed from the completely broken or discontinuous
version of $V_0$. 
The full set of spaces is listed in Table 1, whilst also shows the spaces diagrammatically,
representing scalar DOFs by dots and vector DOFs by arrows.

\subsection{The Recovery Operator}\label{sec:recovery operator}
Here we discuss the details of the recovery operator that
we will use, which is very similar to the `weighted averaging'
operator used in \cite{georgoulis2018recovered}.
Our recovery operator reconstructs $\rho_0\in V_0$ into $
\tilde{V}_1$ using the following procedure.
Let $i$ be a degree of freedom in the space $\tilde{V}_1$.
The value of the field in $\tilde{V}_1$ at $i$ is 
determined to be the value of $\rho_0$ at the location of $i$.
However as $\tilde{V}_1$ is continuous, $i$ may be shared 
between a set of multiple elements $\curlybrackets{e_i}$.
In this case the value in $\tilde{V}_1$ is the average of the
values over $\curlybrackets{e_i}$.
Such an operator is found in \cite{georgoulis2018recovered}
to possess second-order convergence in the $L^2$-norm when
$V_0$ is the discontinuous constant space 
$\mathrm{DG}_0\times\mathrm{DG}_0$
and $\tilde{V}_1$ is $\mathrm{CG}_1\times\mathrm{CG}_1$,
the space of continuous linear functions over cells.
These spaces correspond to those listed in Section 
\ref{sec: example spaces} that we will use for our advection schemes. 
\newR{This operator is intended only for use with fields on flat meshes, and must be extended for transport of vector fields on curved meshes.}\\
\\
However, this operation for $\mathrm{DG}_0\times\mathrm{DG}_0$ 
to $\mathrm{CG}_1\times\mathrm{CG}_1$
does not have second-order convergence when representing
fields with non-zero gradient at the boundaries of the domain.
Whilst the second-order convergence holds for the interior
of the domain, at the boundaries degrees of freedom are shared
by fewer elements and may not necessarily accurately represent
the gradient.
We therefore extend the recovery operation at the boundaries
by finding the field $\rho_1\in \hat{V}_1$ that minimises the curvature
$\int_e |\grad \rho_1|^2 \dx{x}$ over a given element $e$
subject to the constraints:
\begin{enumerate}
\item $\rho_1=\rho_r$ at the interior degrees of freedom
of the element $e$, where $\rho_r\in \tilde{V}_1$
is the first recovered field.
\item $\int_e \rho_1 \dx{x}=\int_e\rho_0 \dx{x}$ for the element
$e$, where $\rho_0\in V_0$ is the original field.
\end{enumerate}
Once this $\rho_1$ has been found, the final recovered field
$\rho_\mathcal{R}\in \tilde{V}_1$ is given by applying the
original `averaging' recovery operator again to $\rho_1$. \\
\\
Finally, we will now show that the specific recovery operator that
we have defined in this section satisfied Assumption \ref{assumption},
i.e. for all $\rho_0\in V_0$ there is some $C>0$ such that 
$||\mathcal{R}\rho_0||\leq C||\rho_0||$, where we take
$V_0=\mathrm{DG}_0\times\mathrm{DG}_0$ and
$\tilde{V}_1=\mathrm{CG}_1\times\mathrm{CG}_1$.
This property is used in Section \ref{sec:stability}.
\setcounter{theorem}{0}
\begin{theorem} \label{thm:recovery_bounded}
Consider the action of the specific recovery operator 
$\mathcal{R}:V_0\to \tilde{V}_1$ defined in
Section \ref{sec:recovery operator} upon a field
$\rho_0\in V_0$ for $V_0=\mathrm{DG}_0\times\mathrm{DG}_0$
and $\tilde{V}_1=\mathrm{CG}_1\times\mathrm{CG}_1$.
There is some $C>0$ such that 
$||\mathcal{R}\rho_0||\leq C||\rho_0||$ 
for all $\rho_0\in V_0$, where $||\cdot ||$ denotes the $L^2$
norm.
\end{theorem}
\begin{proof}
We begin by defining $\rho_\mathcal{R}:=\mathcal{R}\rho_0$.
We consider the $L^2$ norm of $\rho_\mathcal{R}$ over an
individual element $e_i$, $||\rho_\mathcal{R}||^2_{e_i}$.
The element $e_i$ and cells that it shares DOFs with form
a patch $\mathcal{P}_i$, and we introduce a coordinate scaling
$\bm{x}\to \bm{x}/h$ under which $e_i$ becomes $e_i'$,
which has unit area.
If $\curlybrackets{\phi_j(\bm{x})}$ are the $M$ basis functions 
spanning $\tilde{V}_1$, then $\rho_\mathcal{R}$ can be written as
\begin{equation}
\rho_\mathcal{R} = \sum_j^M \rho_{\mathcal{R},j}\phi_j(\bm{x}),
\end{equation}
where $\rho_{\mathcal{R},j}$  is the value of 
$\rho_\mathcal{R}$ at the $j$-th DOF.
We now consider $||\rho_\mathcal{R}||^2_{e_i'}$
\begin{equation}
||\rho_\mathcal{R}||^2_{e_i'} =
\int_{e_i'}\sum_j^M\sum_k^M \rho_{\mathcal{R},j}
\rho_{\mathcal{R},k} \phi_j(h\bm{x}) \phi_k(h\bm{x}) \dx{x}
\equiv ||\bm{\rho_{\mathcal{R}}}||^2_{\Phi',e_i'},
\end{equation}
where $\bm{\rho_{\mathcal{R}}}$ is the vector of values of
$\rho_\mathcal{R}$ at DOFs and 
$||\cdot||_{\Phi'}\equiv ||\bm{y}^T\Phi'\bm{y}||$ denotes the norm 
with the mass matrix
$\Phi':=\int \phi_j(h\bm{x})\phi_k(h\bm{x}) \dx{x}$ 
acting upon some $M$-dimensional vector $\bm{y}$.
From norm equivalence, we know that for some $C>0$,
$\bm{\rho_{\mathcal{R}}}$ evaluated in the $\Phi'$ norm can be
bounded from above by evaluation with the vector norm, and we
can write this as the sum of components
\begin{equation}
||\rho_\mathcal{R}||^2_{e_i'} \leq C_i\sum_{j\in e_i}^M 
\left(\rho_{\mathcal{R},j}\right)^2,
\end{equation}
where $C_i$ is a constant that depends upon the size of the element.
If the $j$-th DOF in $\tilde{V}_1$ is shared between $N_j$ cells,
then $\rho_{\mathcal{R},j}$ is the average values of $\rho_0$
in those cells, and hence $\rho_{\mathcal{R},j} = \frac{1}{N_j}\sum_k^{N_j} \rho_{0,k}$, giving
\begin{equation}
\begin{split}
||\rho_\mathcal{R}||^2_{e_i'} \leq C_i\sum_{j\in e_i'}^M 
\left(\frac{1}{N_j}\sum_k^{N_j} \rho_{0,k}\right)^2
& \leq C_i\sum_{j\in e_i'}^M 
\left(\sum_k^{N_j} \rho_{0,k}\right)^2 \\
& \leq C_i\sum_{j\in e_i'}^M 
\sum_k^{N_j} \left(\rho_{0,k}\right)^2
\leq C_i||\rho_0||_{\mathcal{P}_i'}^2,
\end{split}
\end{equation} 
where the equalities follow as $N_j$ is a positive integer, 
from the Cauchy-Schwarz inequality and from the definition of
the $L^2$ norm in $V_0$.
The constant $C_i$ has absorbed the effect of double-counting of
cells over the patch.
Under some regularity assumptions about the shape of the mesh,
\begin{equation}
||\rho_\mathcal{R}||_{e_i}^2 \leq C^\ast ||\rho_0||^2_{\mathcal{P}_i}
\end{equation}
where $C^{\ast}$ is the maximum of $C_i$ over the mesh.
Now, considering the $L^2$ norm over the whole domain,
\begin{equation}
||\rho_\mathcal{R}||^2 \leq \sum_i ||\rho_\mathcal{R}||^2_{e_i}
\leq \sum_i C_i ||\rho_0||^2_{\mathcal{P}_i} \leq
C||\rho_0||^2,
\end{equation}
where the constant again takes double-counting into account.
This also holds for the procedure at the boundaries, where
the values of the reconstructed field are a linear function of the interior and original field values.
Thus we arrive at the conclusion that for some $C>0$,
\begin{equation}
||\rho_\mathcal{R}|| \leq C ||\rho_0||.
\end{equation}
\end{proof}

\subsection{Limiting}\label{sec:limiting}
In numerical weather or climate models, there may be many additional prognostic
variables representing moisture or chemical species.
These variables will typically lie in \newR{either the same space as the density $\rho$ or the potential temperature $\theta$.
In this paper we will consider only moisture variables, which 
will lie in the space of $\theta$, which can simplify the
thermodynamics associated with phase changes.
However it may come at the cost of sacrificing conservation
of the mass of water, although this could be remedied by solving
the transport equation in `conservative' form 
(\ref{eqn:conservative form}) for $r\rho$ rather than
in `advective' form for tracer $r$.} \\
\\
The continuous equations describing the advection of these 
\newP{tracer} variables have monotonicity
and shape-preserving properties; however the discrete representation may not replicate
these properties, which may lead to unphysical solutions such as negative concentrations.
This can be avoided by the application of slope limiters. \\
\\
In the recovered scheme, both the advection operator $\mathcal{A}$ and final projection
operator $\mathcal{P}$ may produce spurious overshoots and undershoots,
and so need limiting.
In the case of the projection operator, we do this by using the second projection operator
$\mathcal{P}_B$.
\newB{This prevents the formation of new maxima and minima as it is composed of two bounded operations: the projection into the broken lower-order space and then the recovery of continuity.}
For the set of spaces proposed in Section \ref{sec: example spaces}, to limit the advection
operator $\mathcal{A}$, we use the vertex-based limiter outlined in \cite{kuzmin2010vertex}.
\newB{This limiter divides the field in each element into a constant mean part and a linear perturbation.
Considering the values of neighbouring elements at shared vertices gives upper and lower bounds.
The limited field is then the mean part plus a constant times the perturbation, so that the field remains bounded.}
This limiter is applied to the field before the advection operator begins, and after each stage
of $\mathcal{A}$. 

\section{Properties of the Numerical Scheme}
\label{sec:properties}
\subsection{Stability} \label{sec:stability}
Here we will show the stability of the `recovered space' scheme, following a 
similar argument to that used in \cite{cotter2016embedded}.
First, we will need the following Lemma.
\begin{lemma} \label{lemma}
Let the operator $\mathcal{J}:V_0\to V_1$ 
be defined by
\begin{equation}
\mathcal{J}:=\mathcal{I}(\mathcal{R}-\hat{\mathcal{P}}\mathcal{R}+1),
\end{equation} 
so that the `recovered space' scheme can be written as
\begin{equation}
\rho_0^{n+1}=\mathcal{PAJ}\rho_0^n.
\end{equation}
Denote by $||\cdot||$ the $L^2$ norm. 
Then $||\mathcal{J}\rho_0||\leq \kappa ||\rho_0||$ for some $\kappa>0$ 
for all $\rho_0\in V_0$.
\end{lemma}
\begin{proof} From the definition of $\mathcal{J}$,
\begin{equation}
||\mathcal{J}\rho_0|| = ||\rho_0 + \mathcal{R}\rho_0 - \hat{\mathcal{P}}
\mathcal{R}\rho_0||.
\end{equation}
By applying the triangle inequality,
\begin{equation}
||\mathcal{J}\rho_0||\leq ||\rho_0|| + ||\mathcal{R}\rho_0|| +
||\hat{\mathcal{P}}\mathcal{R}\rho_0||.
\end{equation}
We will now inspect the $||\hat{\mathcal{P}}\mathcal{R}\rho_0||$ term.
The definition of $\hat{\mathcal{P}}$ is that 
$\int_\Omega \hat{\psi}\tilde{\rho}\dx{x}=
\int_\Omega \hat{\psi}\hat{\mathcal{P}}\tilde{\rho}\dx{x}$ 
for all $\hat{\psi}\in \hat{V}_0$, where $\tilde{\rho}\in\tilde{V}$.
Since $\hat{\mathcal{P}}\tilde{\rho}\in\hat{V}_0 $, then it must be true
 that
\begin{equation} \label{eqn:project_special_case}
\int_\Omega (\hat{\mathcal{P}}\tilde{\rho})^2\dx{x} = 
\int_\Omega (\hat{\mathcal{P}}\tilde{\rho})\tilde{\rho}\dx{x}.
\end{equation}
Now we consider the integral 
$\int_\Omega (\hat{\mathcal{P}}\tilde{\rho} -
\tilde{\rho})^2\dx{x}$, which cannot be negative.
Expanding this integral and using the result (\ref{eqn:project_special_case})
gives $\int_{\Omega}\tilde{\rho}^2\dx{x}\geq 
\int _\Omega (\hat{\mathcal{P}}\tilde{\rho})^2\dx{x}$, in other words that
$||\hat{\mathcal{P}}\mathcal{R}\rho_0||\leq||\mathcal{R}\rho_0||$.
Hence, returning to considering $||\mathcal{J}\rho_0||$, we obtain
\begin{equation}
||\mathcal{J}\rho_0||\leq ||\rho_0|| + 2||\mathcal{R}\rho_0||.
\end{equation}
Finally, we use the property of $\mathcal{R}$
that $||\mathcal{R}\rho_0||\leq C ||\rho_0||$ for some $C> 0$,
and so letting $\kappa=1+2C$ we arrive at
\begin{equation}
||\mathcal{J}\rho_0||\leq \kappa ||\rho_0||.
\end{equation}
This completes the proof.
\end{proof}

\begin{theorem}
Let the advection operator $\mathcal{A}$ have a
stability constant $\alpha$, such that 
\begin{equation}
||\mathcal{A}||:=\sup_{\rho_1\in V_1, ||\rho_1||>0}
\frac{||\mathcal{A}\rho_1||}{||\rho_1||}\leq \alpha.
\end{equation}
Then the stability constant $\alpha^*$ of the `recovered space' scheme on $V_0$ satisfies $\alpha^*=\kappa\alpha$ for some constant $\kappa$.
\end{theorem}

\begin{proof}
Since from Lemma \ref{lemma},
$||\mathcal{J}\rho_0||\leq \kappa ||\rho_0||$ for some
$\kappa>1$, 
\begin{equation}
\sup_{\rho_0\in V_0, ||\rho_0||>0}\frac{||\mathcal{PAJ}\rho_0||}{||\rho_0||}
\leq \sup_{\rho_0\in V_0, ||\rho_0||>0} 
\kappa \frac{||\mathcal{PAJ}\rho_0||}{||\mathcal{J}\rho_0||}.
\end{equation}
As $V_0\subset V_1$, the supremum over elements in $V_1$ cannot be 
smaller than the supremum over elements in $V_0$.
Recognising that $\mathcal{J}\rho_0\in V_1$,
\begin{equation}
\sup_{\rho_0\in V_0, ||\rho_0||>0} 
\frac{||\mathcal{PAJ}\rho_0||}{||\mathcal{J}\rho_0||} \leq
\sup_{\rho_1\in V_1, ||\rho_1||>0} 
\frac{||\mathcal{PA}\rho_1||}{||\rho_1||}.
\end{equation}
For the final step, we must consider both cases
$\mathcal{P}_A$ and $\mathcal{P}_B$ for the projection operator.
In the case of $\mathcal{P}_A$, we can use a similar argument to that
of the projection operator in Lemma \ref{lemma} to obtain that
$||\mathcal{PA}\rho_1||\leq ||\mathcal{A}\rho_1||$.
For $\mathcal{P}_B=\mathcal{P}_R\mathcal{P}_I$, each step maps
a function into a space that is smaller; i.e. $\hat{V}_0\subset V_1$ and
$V_0\subset \hat{V}_0$, so that the supremum of $||\mathcal{PA}\rho_1||$
must be smaller than the supremum of $||\mathcal{A}\rho_1||$.
In both cases we obtain that 
\begin{equation}
\sup_{\rho_1\in V_1, ||\rho_1||>0} 
\frac{||\mathcal{PA}\rho_1||}{||\rho_1||} \leq
\sup_{\rho_1\in V_1, ||\rho_1||>0} 
\frac{||\mathcal{A}\rho_1||}{||\rho_1||}\leq \alpha,
\end{equation}
where the final inequality defines the stability of $\mathcal{A}$.
Combining these arguments together gives
\begin{equation}
\sup_{\rho_0\in V_0, ||\rho_0||>0}\frac{||\mathcal{PAJ}\rho_0||}{||\rho_0||}
\leq \kappa \alpha ,
\end{equation}
and so the stability constant of the `recovered space' 
scheme is $\alpha^*=\kappa\alpha$.
\end{proof}

\subsection{Von Neumann Analysis}\label{sec: von neumann}
Now we will attempt to identify the stability constant for three one-dimensional
examples, by performing Von Neumann stability analysis.
This can also be used with the Courant-Friedrich-Lewy (CFL) condition
to give upper limits to stable Courant numbers.\\
\\
The three cases that will be considered are that of $V_0 = \mathrm{DG}_0$
(which might represent the advection of $\rho$), and the two cases of
$V_0=\mathrm{CG}_1$ with $\mathcal{P}_A$ and $\mathcal{P}_B$ as the
projection operators (for advection of velocity and moisture
respectively).
The same advection operator \newR{discretising the advective form 
(\ref{eqn:advective}) of the transport equation} will be used for all three cases,
with the advection taking place in $V_1=\mathrm{DG}_1$. \\
\\
In each case we consider a periodic domain of length $L$, divided into $N$ cells, each 
of length $\Delta x$.
We will make the assumption that our function $\rho_n(x)$ at the $n$-th
time step can be written as a sum of Fourier modes,
\begin{equation}
\rho^n(x) = \sum_k A_k^n e^{ikx}.
\end{equation}
Then for the $k$-th mode $\rho_{n,k}(x+\Delta x)=\rho_{n,k}(x)e^{ik\Delta x}$.
\\
\\
Three spaces are relevant to this analysis:
DG$_0$, which is piecewise constant and whose DOFs are in the centre of cells;
CG$_1$, which is continuous piecewise linear and has one DOF per cell at the cell boundary;
and DG$_1$, which is linear within a cell but discontinuous between cells and has
two DOFs per cell -- one at each cell boundary.
These spaces are shown in Figure 1.
\begin{figure} 
\centering
\includegraphics[scale=0.1]{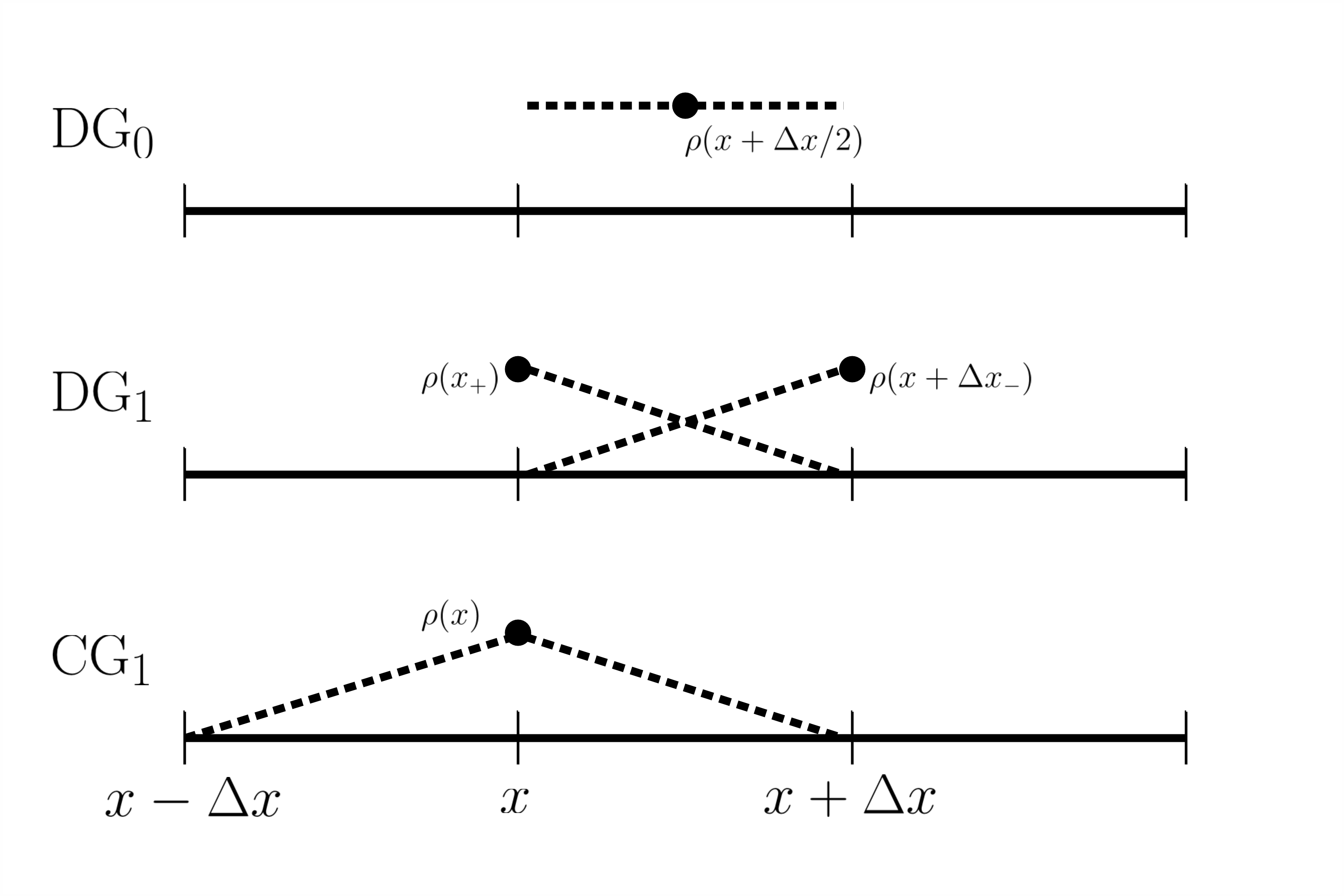} 
\caption*{\small{\textbf{Figure 1}:
The degrees of freedom and basis functions for each of the spaces
used in the Von Neumann analysis.}} \label{basis_functions}
\end{figure}

\subsubsection{Advection Operator}
First we described the advection operator $\mathcal{A}_k$ acting upon the
$k$-th mode of a function in $\mathrm{DG}_1$.
In each cell, the function can be described by two components: evaluation
of the field at each cell boundary.
For the advection, we use a simple upwinding scheme with a forward Euler
time discretisation, within the framework of a three-step Runge-Kutta scheme.
We describe the action of a single forward Euler step with the operator $\mathcal{L}_k$.
This is determined by discretising the one dimensional
advection equation with constant $u>0$,
\begin{equation}
\pfrac{q}{t}+u\pfrac{q}{x}=0,
\end{equation}
for $q\in V_1$ by integrating with the test function $\psi\in V_1$.
This gives
\begin{equation}
\int_0^L \psi q_{n+1}\dx{x} = 
\int_0^L \psi q_{n} \dx{x} + u\Delta t\int_0^L \pfrac{\psi}{x}q_{n}\dx{x} 
-u\Delta t\sum_j \llbracket \psi q_n \rrbracket_j,
\end{equation}
where $\llbracket \cdot \rrbracket_j$ denotes the jump in field between the
 $j$-th cell and the $(j+1)$-th cell.
Making the assumption that $q(x)=q(x+\Delta x)e^{-ik \Delta x}$, and
using that $q$ is piecewise linear, we can write down a representation of
$\mathcal{L}_k$ for the degrees of freedom on either side of a given cell:
\begin{equation}
\mathcal{L}_k = \begin{pmatrix}
1 - 3c & 4ce^{-ik\Delta x} - c \\
3c & 1 - c - 2ce^{-ik\Delta x}
\end{pmatrix}.
\end{equation}
We then obtain the full advection operator 
by using the three-step Runge-Kutta scheme
outlined in \cite{shu1988efficient}:
\begin{align}
q^{(1)}& := q_{n,k} + \mathcal{L}_k q_{n,k}, \\
q^{(2)} & := \frac{3}{4}q_{n,k} + \frac{1}{4}q^{(1)} 
+ \frac{1}{4}\mathcal{L}_k q^{(1)}, \\
q_{n+1,k}&  =  \frac{1}{3}q_{n,k} + \frac{2}{3}q^{(2)}
+\frac{2}{3}\mathcal{L}_k q^{(2)}.
\end{align}
The overall advection operator is then
\begin{equation}
\mathcal{A}_k =\mathds{1}+ \mathcal{L}_k+\frac{1}{2}\mathcal{L}_k^2+\frac{1}{6}\mathcal{L}_k^3,
\end{equation}
where $\mathds{1}$ is the identity operator.
We omit the matrix representation of $\mathcal{A}_k$ here for brevity. 

\subsubsection{Case A: $\mathrm{DG_0}$}
This represents the advection of density $\rho$, 
or the velocity $\bm{v}$ perpendicular to its direction.
The set of spaces $\lbrace V_0, V_1, \tilde{V}_1, \hat{V}_0\rbrace$
is $\lbrace\mathrm{DG}_0, \mathrm{DG}_1, \mathrm{CG}_1, \mathrm{DG}_0\rbrace$. \\
\\
Then, for a given cell and Fourier mode, the operators can be represented in
the following matrix forms
\begin{equation}
\hat{\mathcal{P}}=\mathcal{P}=
\frac{1}{2}\begin{pmatrix}
1 & 1
\end{pmatrix}, \ \ \ \ \ 
\mathcal{R}_k= \frac{1}{2}
\begin{pmatrix}
1+e^{-ik\Delta x} \\
e^{ik\Delta x} + 1
\end{pmatrix}, \ \ \ \ \ 
\mathcal{I} = 
\begin{pmatrix}
1 \\ 1
\end{pmatrix}.
\end{equation}
Combining these operators, the advection scheme for the 
$k$-th mode of $\rho_n(x)$ is  then expressed as
\begin{equation}
\rho_{n+1,k}=\frac{1}{4}\begin{pmatrix}
1 & 1
\end{pmatrix}\mathcal{A}_k
\begin{pmatrix}
2 - i\sin(k\Delta x) \\
2 + i\sin(k\Delta x)
\end{pmatrix} \rho_{n,k}.
\end{equation}
Following the analysis through and writing
$\phi=k\Delta x$ gives a stability constant
\begin{equation}\label{eqn:A_k1}
\begin{split}
& |A_k|^2 = c^2 \left[
c^3 \left(\frac{13}{4}\cos\phi -\frac{5}{3}\cos 2\phi
+\frac{1}{12}\cos 3\phi + \frac{1}{6}\cos 4\phi - \frac{11}{6}\right)
\right. \\
& +  c^2 \left(\frac{29}{12}\sin\phi - \frac{5}{3}\sin 2\phi
+\frac{1}{12}\sin 3\phi + \frac{1}{6}\sin 4\phi
-\frac{7}{4}\cos\phi + \frac{3}{4}\cos 2\phi
-\frac{1}{4}\cos 3\phi + \frac{5}{4}\right) \\
& + c\left(-\frac{3}{4}\sin\phi + \frac{3}{4}\sin 2\phi-
\frac{1}{4}\sin 3\phi - \cos\phi +\frac{1}{4}\cos 2\phi
+\frac{3}{4}\right)
\left. -\frac{3}{2}\sin \phi + \frac{1}{4}\sin^2 2\phi 
- 1\right]^2
\end{split}
\end{equation}

\subsubsection{Case B: $\mathrm{CG_1}$ with $\mathcal{P}=\mathcal{P}_A$}
In this case the set of spaces $\lbrace V_0, V_1, \tilde{V}_1, \hat{V}_0\rbrace$
is $\lbrace\mathrm{CG}_1, \mathrm{DG}_1, \mathrm{CG}_1, \mathrm{DG}_1\rbrace$.
This describes advection of velocity parallel to its direction, or of potential temperature
without bounding the final projection step.
The operators can be represented by
\begin{equation}
\mathcal{P}_A=
\frac{1}{4+2\cos k\Delta x}\begin{pmatrix}
2+e^{-ik\Delta x} & 
1+2e^{-ik\Delta x}
\end{pmatrix}, \ \ \ \ \ 
\hat{\mathcal{P}}=\mathcal{I} = 
\begin{pmatrix}
1 \\ e^{ik\Delta x}
\end{pmatrix},
\end{equation}
where the projection operator $\mathcal{P}_A$
has been determined by solving equation (\ref{eqn:projectionA}).
The recovery operator is the identity, and since the injection $\mathcal{I}$
and the projection $\hat{\mathcal{P}}$ are equivalent in this case
the scheme acting upon $v_n$ becomes 
$v_{n+1}=\mathcal{P}_A\mathcal{A}\mathcal{I}v_n$.
Following through the analysis gives
\begin{equation}\label{eqn:A_k2}
\begin{split}
& |A_k|^2 = \left(\frac{c}{\cos\phi + 2}\right)^2
\left(c^2 \cos\phi - c^2 + 3\right)^2\sin^2\phi \\
& + \left(\frac{c}{\cos\phi + 2}\right)^2
 \left(-2c^3\cos\phi + \frac{1}{2}c^3\cos 2\phi +
\frac{3}{2}c^3 + 3c^2\cos\phi -3c^2+\cos\phi + 2 \right)^2
\end{split}
\end{equation}
\subsubsection{Case C: $\mathrm{CG_1}$ with $\mathcal{P}=\mathcal{P}_B$}
For this case, the set of spaces are the same as in the second case.
The only difference is that the projection operator $\mathcal{P}$ is now
$\mathcal{P}_B=\mathcal{P}_R\mathcal{P}_I$.
As $V_1=\hat{V}_0=\mathrm{DG}_1$, the interpolation $\mathcal{P}_I$ is the identity,
and $\mathcal{P}_B=\mathrm{P}_R$.
The operators are
\begin{equation}
\mathcal{P}_B=
\frac{1}{2}\begin{pmatrix}
1 & 
e^{-ik\Delta x}
\end{pmatrix}, \ \ \ \ \ 
\hat{\mathcal{P}}=\mathcal{I} = 
\begin{pmatrix}
1 \\ e^{ik\Delta x}
\end{pmatrix},
\end{equation}
which gives leads to the amplification factor
\begin{equation} \label{eqn:A_k3}
\begin{split}
& |A_k|^2= \left(\frac{2}{3}c^3\sin\phi
+\frac{1}{6}c^3\sin 2\phi - \frac{1}{3}c^3\sin 3\phi
-c^2\sin\phi + \frac{1}{2}c^2\sin 2\phi - c\sin\phi \right)^2
\\
& + \left(-\frac{7}{3}c^3\cos\phi -\frac{1}{6}c^3\cos 2\phi +
\frac{1}{3}c^3\cos 3\phi + \frac{13}{6}c^3 + 3c^2\cos\phi
-\frac{1}{2}c^2\cos 2\phi-\frac{5}{2}c^2+1\right)^2
\end{split}
\end{equation}

\subsection{Critical Courant Numbers}
The Courant-Friedrich-Lewy (CFL) criterion says that an
advection scheme with amplification factor $|A_k|>1$
may not be stable.
The critical Courant number $c^*$ is the lowest Courant number
$c=u\Delta t/\Delta x$ such that the amplification factor
is greater than unity.
We numerically measured the critical Courant numbers for the
three cases laid out in Section \ref{sec: von neumann}, and these
are displayed in Table 2.
Although case C has a significantly lower critical Courant 
number, the intention is to run this scheme with a limiter
\newR{or with a subcycling time discretisation, allowing it
to be used at higher Courant numbers}.
\newR{Instances of work using these kind of limiters are 
\cite{cotter2016embedded} and \cite{kuzmin2010vertex}.} \\
\\
\newR{Examples of critical Courant numbers for other upwinding schemes
can be found in Table 2.2 of \cite{cockburn2001runge}.
The most relevant comparison that can be made from this
is to that of polynomials of degree 1 with a Runge-Kutta method of order 3,
which has a critical Courant number of 0.409.
A space of discontinuous linear polynomials has the same number of degrees of freedom
as a space of discontinuous constants but with half the grid size, and thus improvements
are made if the critical Courant number is more than twice that of the transport scheme for the
linear functions.
We do therefore observe that the critical Courant numbers for cases A and B are improvements
on the discontinuous upwinding scheme applied just to discontinuous linear functions.}
\begin{table}[h!]
\centering
\begin{tabular}{c|c|c|c}
Case & A & B & C \\
\hline
$c^*$ & 0.8506 & 0.9930 & 0.3625
\end{tabular}
\caption*{\small{\textbf{Table 2}:
The critical Courant numbers for the three cases of the advection
scheme outlined in Section \ref{sec: von neumann}.
}}
\end{table}

\section{Numerical Tests}
\label{sec:numerics}
\newP{The numerical implementation of this scheme was performed using
the Firedrake software of \cite{rathgeber2017firedrake}, and
relied heavily on the tensor product element functionality on
extruded meshes, descriptions of which can be found in
\cite{mcrae2016automated}, \cite{bercea2016structure} and \cite{homolya2016parallel}.}
\subsection{Numerical Accuracy}
To verify the numerical accuracy of the scheme, we performed
a series of convergence tests.
The aim is to find how the error due to advection changes
with the grid spacing $\Delta x$.
We used tests that have an analytic solution
in the limit that $\Delta x\to 0$, and compare the final
advected profile $q$ with the `true' profile $q_h$, which
is the analytic solution projected into the relevant function
space.
This gives an error $||q-q_h||$ (where $||\cdot ||$ denotes the $L^2$ norm) which is calculated for 
the same problem at different 
resolutions, and the errors are plotted as a function
of the grid spacing $\Delta x$.
The order of the numerical accuracy is the number $n$ such that
$||q-q_h|| \sim \mathcal{O}(\Delta x^n)$, which can 
measured from the \newR{slope} of a plot of $\ln(||q-q_h||)$
against $\ln(\Delta x)$.
For simplicity, the tests we used are designed so that the `true'
profile is the same as the initial condition. \\
\\
The initial conditions were obtained by pointwise evaluation
of the expressions into higher order fields (we used CG$_3$).
These were then projected into the correct fields.
The advecting velocity used lay in the RT$_1$ space.
To mimic how the scheme might be used in a numerical weather
model, we performed some of the tests on the different sets
of spaces laid out in Section \ref{sec: example spaces}
and the configurations described in Section 
\ref{sec: von neumann}.
Each set of spaces is labelled by the variable name in Table 1,
\newB{with $\curlybrackets{V_0, V_1, \tilde{V}_1, \hat{V}_0}$ 
for the fields $\rho$, $\bm{v}$, $\theta$ and $r$ given by
\begin{align}
\rho & \in \curlybrackets{\mathrm{DG}_0\times \mathrm{DG}_0,
\mathrm{DG}_1\times \mathrm{DG}_1,
\mathrm{CG}_1\times \mathrm{CG}_1,
\mathrm{DG}_0\times \mathrm{DG}_0}, \\
\bm{v} & \in \curlybrackets{\mathrm{RT}_1,
\mathbf{DG}_1\times \mathbf{DG}_1,
\mathbf{CG}_1\times \mathbf{CG}_1,
\mathrm{broken} \ \mathrm{RT}_1}, \\
\theta & \in \curlybrackets{\mathrm{DG}_0\times \mathrm{CG}_1,
\mathrm{DG}_1\times \mathrm{DG}_1,
\mathrm{CG}_1\times \mathrm{CG}_1,
\mathrm{DG}_0\times \mathrm{DG}_1}, \\
r & \in \curlybrackets{\mathrm{DG}_0\times \mathrm{CG}_1,
\mathrm{DG}_1\times \mathrm{DG}_1,
\mathrm{CG}_1\times \mathrm{CG}_1,
\mathrm{DG}_0\times \mathrm{DG}_1},
\end{align}
where the bold font recognises that the space has vector valued nodes.}
While the scheme labelled $\theta$ uses the projection operator
$\mathcal{P}_A$, the scheme labelled $r$ represents a moisture
variable, so uses the same spaces
as $\theta$ but the projection operator $\mathcal{P}_B$
and the vertex-based limiter of \cite{kuzmin2010vertex}.
\newR{All tests solve the transport equation in the `advective' form
of (\ref{eqn:advective}).} \\
\\
The first three tests involve advection around a 2D domain 
representing a vertical slice that is periodic along
its side edges but rigid walls at the top and bottom.
The final test is performed over the surface of a sphere. 
All the vertical slice tests use a domain of height and width
1 m and advect the profile with time steps of
$\Delta t=10^{-4}$ s for a total time of $T=1$ s. 
\subsubsection{Rotational Convergence Test} \label{sec:first convergence}
The first test involves a rigid body rotation of a Gaussian
profile around the centre of the domain.
Using $x$ and $z$ as the horizontal and vertical coordinates,
defining $r^2=(x-x_0)^2+(z-z_0)^2$ for $x_0=0.375$ m, $z_0=0.5$ m
and using $r_0=1/8$ m, the initial condition used for all
fields was
\begin{equation}
q = e^{-(r/r_0)^2}.
\end{equation}
For the velocity variable, this initial profile was used
for each component of the field.
The advecting velocity is generated from a stream function 
$\psi$ via  $\bm{v}=(-\partial_z\psi, \partial_x\psi)$.
Defining $r_v^2 = (x-x_v)^2+(z-z_v)^2$  with $x_v=0.5$ m and
$z_v = 0.5$ m, the stream function used was:
\begin{equation}
\bm{\psi}(\bm{x}) =
\left\lbrace 
\begin{matrix}
\pi(r_v^2 - 0.5), & r_v < r_1, \\
Ar_v^2+Br_v+C,  & r_1 \leq r_v < r_2, \\
Ar_2^2+Br_2+C, & r_v \geq r_2.
\end{matrix} \right. 
\end{equation}
This is designed to be a rigid body rotation for $r_v<r_1$, with
no velocity for $r_v\geq r_2$ to prevent spurious noise
being generated from the edge of the domain.
The stream function and its derivative vary smoothly for
$r_1 \leq r_v < r_2$.
We use $r_1=0.48$ m and $r_2=0.5$ m, with 
$A = \pi r_1 / (r_1 - r_2)$, $B=-2Ar_2$ and 
$C=\pi(r_1^2-0.5)-Ar_1^2-Br_1$.
Results showing second order numerical accuracy can be found
in Figure 4 (left).
\newB{Initial and final fields for the density in the lowest resolution run ($\Delta x = 0.01$ m)
are displayed in Figure 2.}
\begin{figure}[h!]
\centering
\begin{subfigure}{.48\textwidth}
\centering
\includegraphics[width=\textwidth]{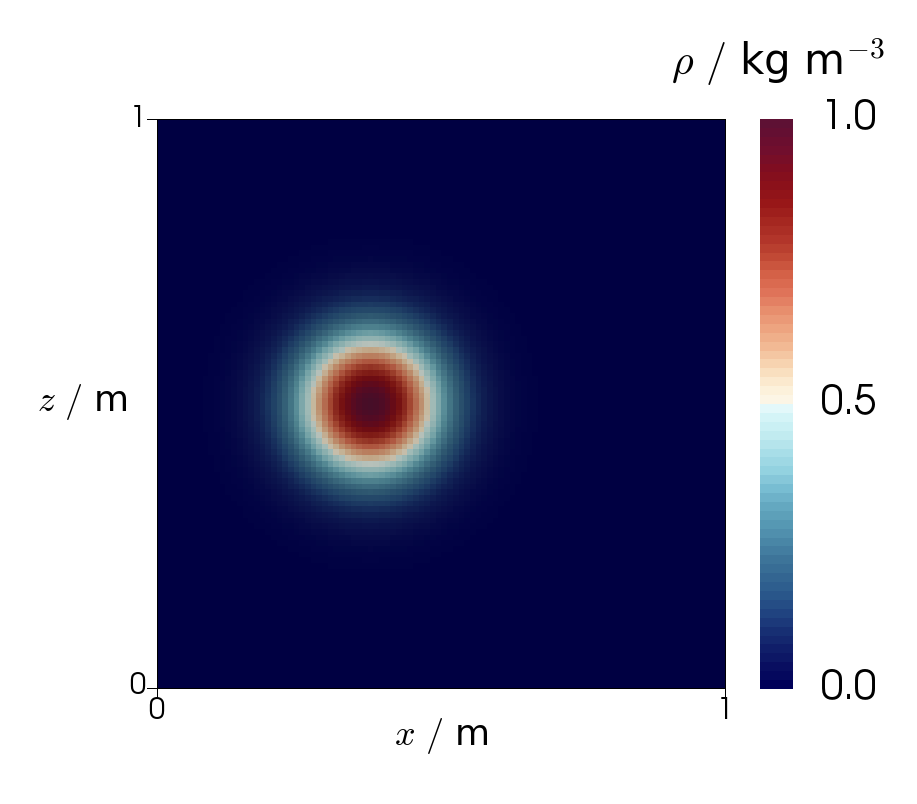}
\end{subfigure}
~~
\begin{subfigure}{.48\textwidth}
\centering
\includegraphics[width=\textwidth]{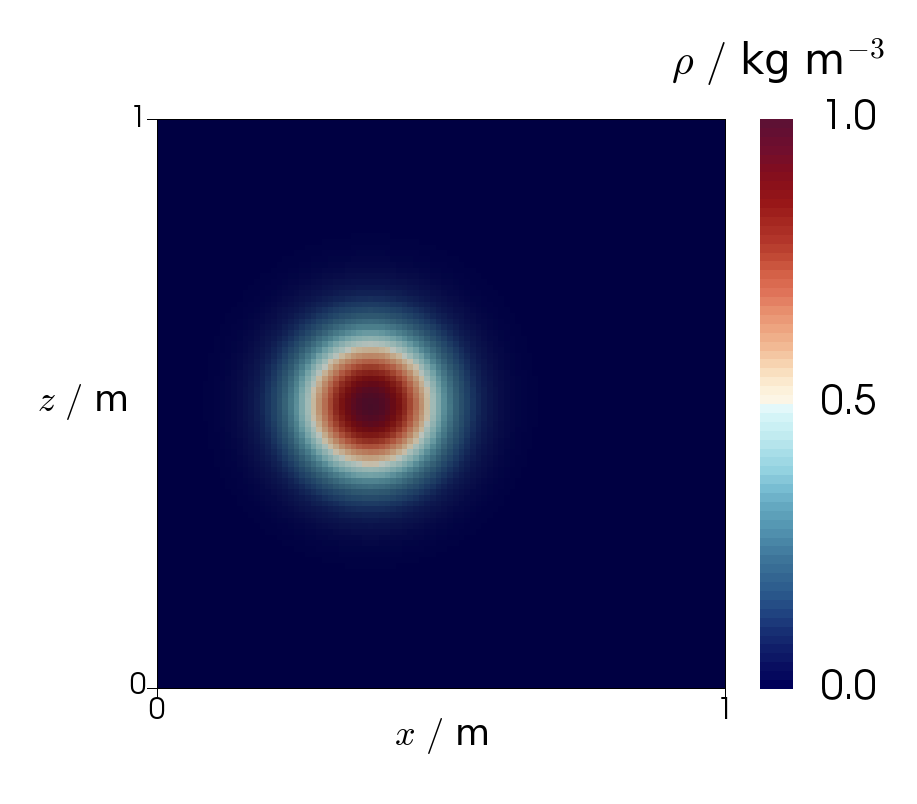}
\end{subfigure}
\caption*{\small{\newB{\textbf{Figure 2}:
The initial (left) and final (right) fields in the $\mathrm{DG}_0\times\mathrm{DG}_0$ space from the rotational convergence
test of Section \ref{sec:first convergence}, showing the field labelled
$\rho$ in Figure 4 (left) at the lowest resolution ($\Delta x=0.01$ m).
Almost no difference is visible between the fields.}
}}
\end{figure}
\subsubsection{Deformational Convergence Test}\label{sec:second convergence}
The second test is a more challenging convergence test,
based on the deformational flow
experiment described in \cite{cotter2016embedded}.
The initial profiles were the same as used in the rotational
advection test, but with $x_0=0.5$ m.
The advecting velocity was that of \cite{cotter2016embedded}:
\begin{equation}
\bm{v}(\bm{x}, t) = 
\begin{pmatrix}
1 - 5(0.5 -t)\sin(2\pi(x-t))\cos(\pi z) \\
5(0.5 -t)\cos(2\pi(x-t))\sin(\pi z)
\end{pmatrix}.
\end{equation}
Figure 4 (right)
plots the results of this test, with each variable measuring
second order numerical accuracy.
\newB{Initial and final fields for the density in the lowest resolution run ($\Delta x = 0.01$ m)
are displayed in Figure 3.}
\begin{figure}[h!]
\centering
\begin{subfigure}{.48\textwidth}
\centering
\includegraphics[width=\textwidth]{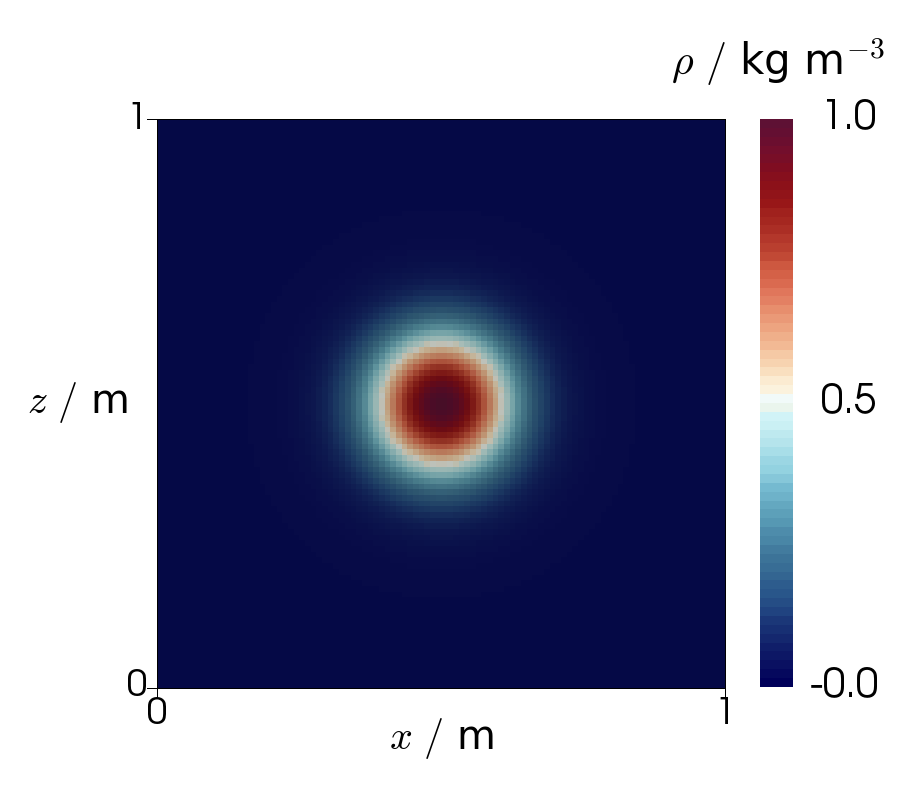}
\end{subfigure}
~~
\begin{subfigure}{.48\textwidth}
\centering
\includegraphics[width=\textwidth]{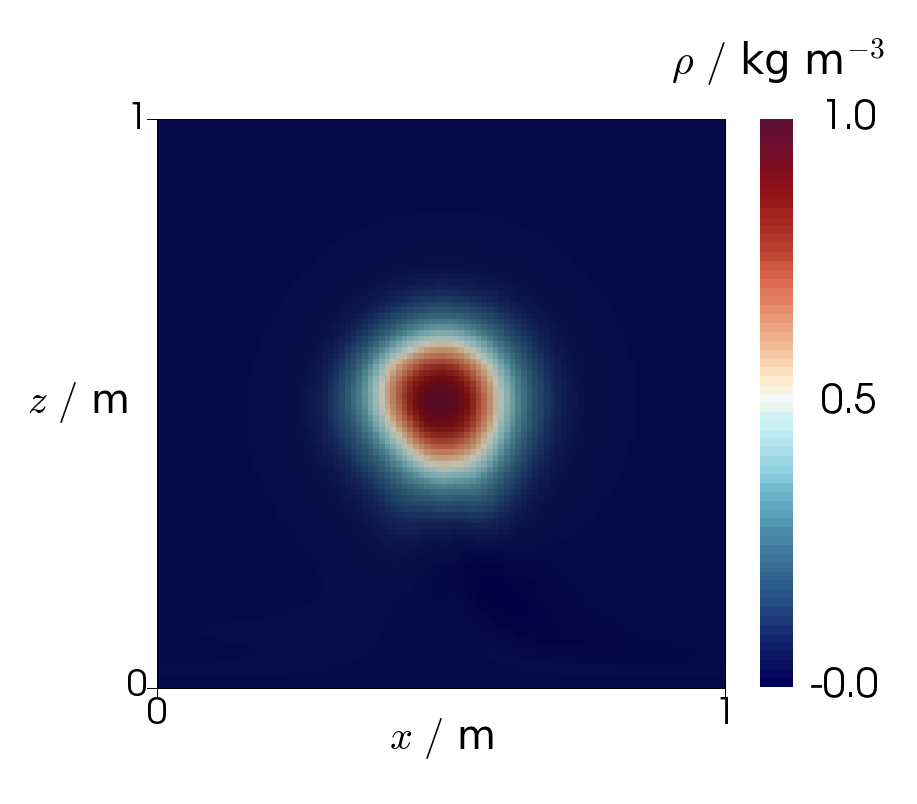}
\end{subfigure}
\caption*{\small{\newB{\textbf{Figure 3}:
The initial (left) and final (right) fields in the $\mathrm{DG}_0\times\mathrm{DG}_0$ space from the deformational convergence
test of Section \ref{sec:second convergence}, showing the field labelled
$\rho$ in Figure 4 (right) at the lowest resolution ($\Delta x=0.01$ m).}
}}
\end{figure}
\begin{figure}[h!]
\centering
\begin{subfigure}{.48\textwidth}
\centering
\includegraphics[width=\textwidth]{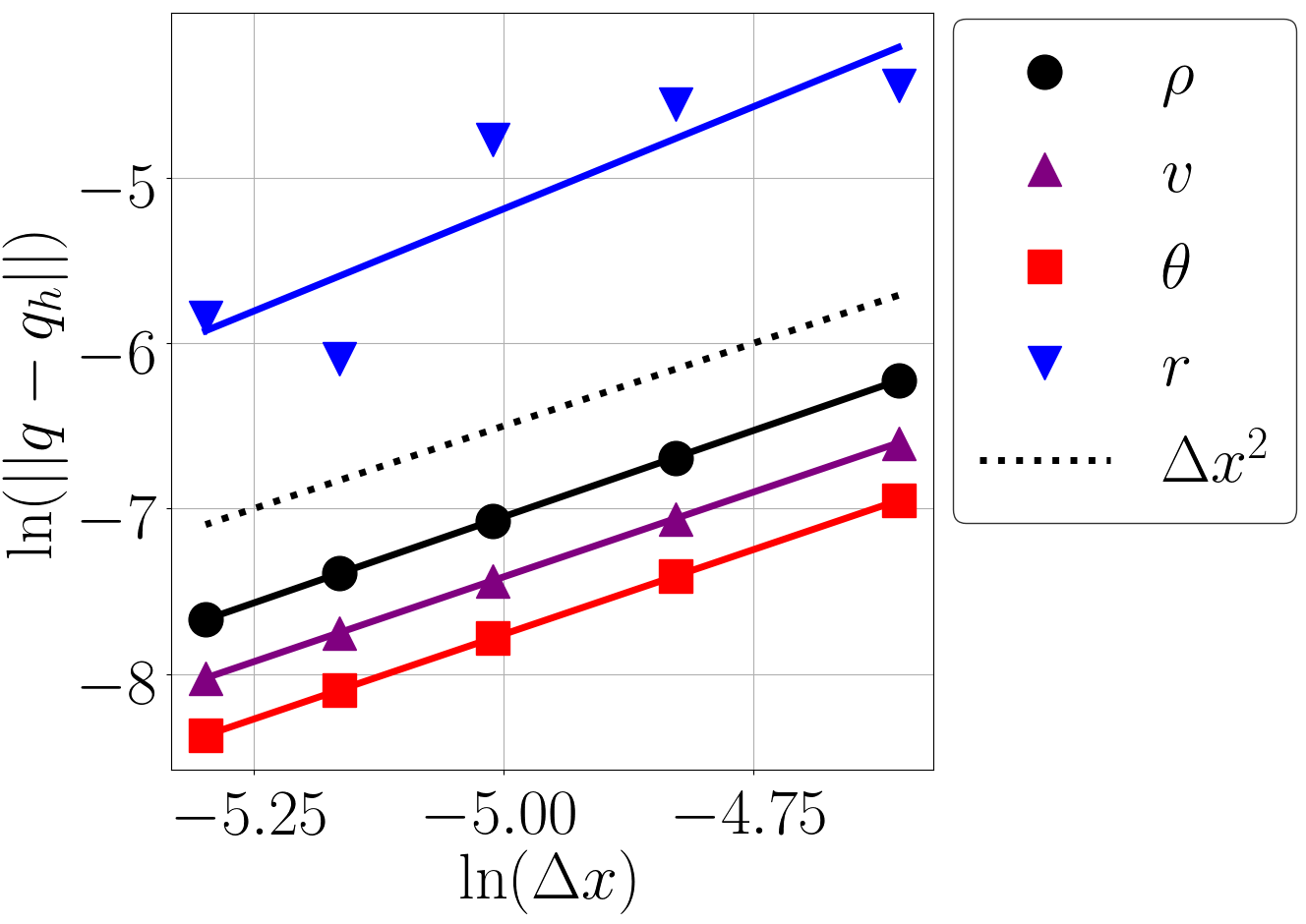}
\end{subfigure}
~~
\begin{subfigure}{.48\textwidth}
\centering
\includegraphics[width=\textwidth]{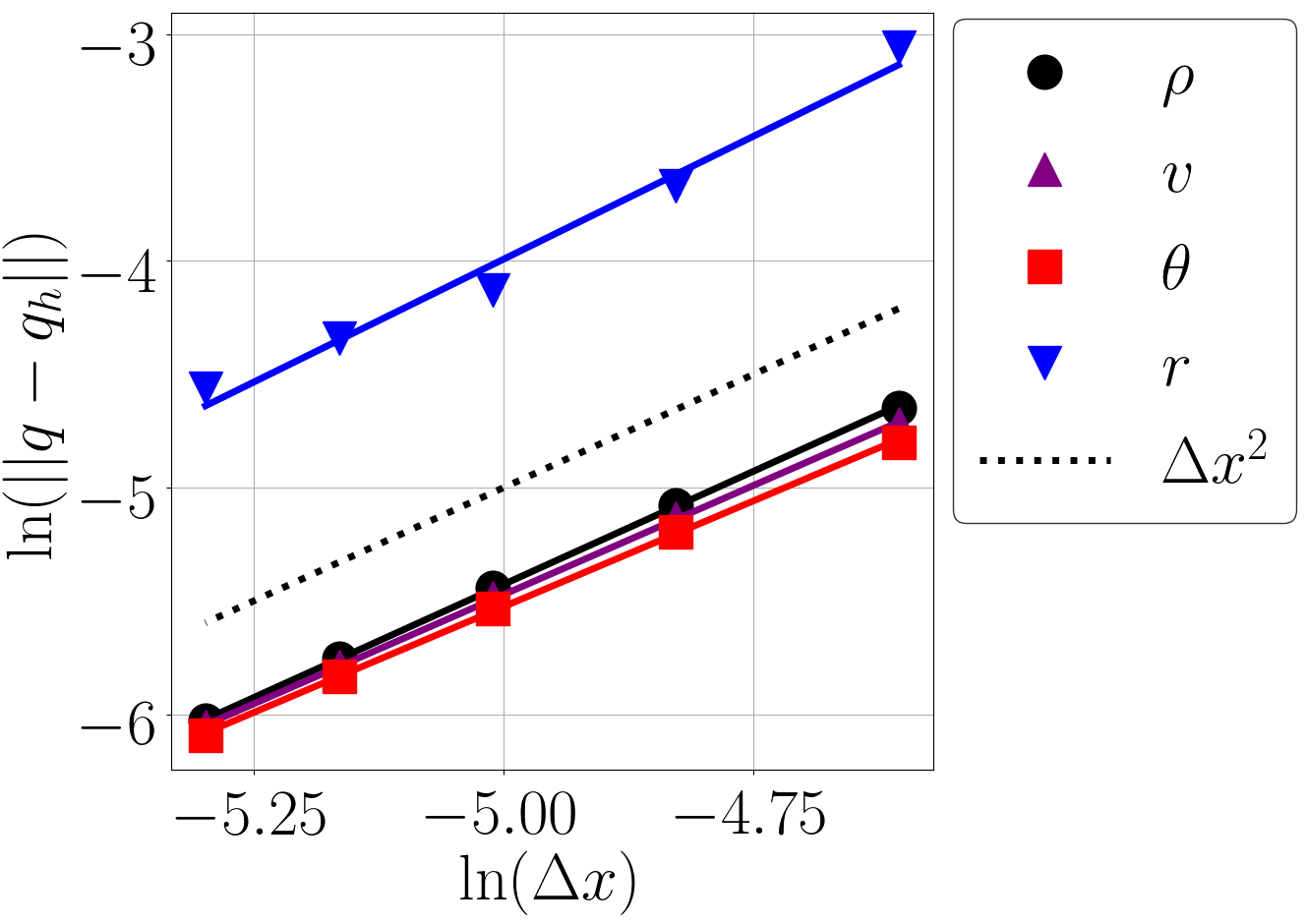}
\end{subfigure}
\caption*{\small{\textbf{Figure 4}:
Results from convergence tests for the recovered space scheme,
plotting, as a function of grid spacing $\Delta x$, the error in 
an advected solution $q$ against the true solution $q_h$.
\newR{The different lines labelled $\rho$, $v$ and $\theta$ represent performing the test in each of the different sets of spaces laid out in Section \ref{sec: example spaces} and using the projection operator $\mathcal{P}_A$, whilst the line labelled $r$ represents advection with the same spaces as $\theta$, but with the projection operator $\mathcal{P}_B$.}
(Left) The test represents a rigid body rotation.
In all cases, the \newR{slopes} are around $2$, indicating
second order numerical accuracy.
(Right) A more difficult convergence test featuring deformational
flow. Accuracy is approaching second order for each case.
}}
\end{figure}
\subsubsection{Boundary Convergence Test} \label{sec:third convergence}
\begin{figure}[h!]
\centering
\begin{subfigure}{.48\textwidth}
\centering
\includegraphics[width=\textwidth]{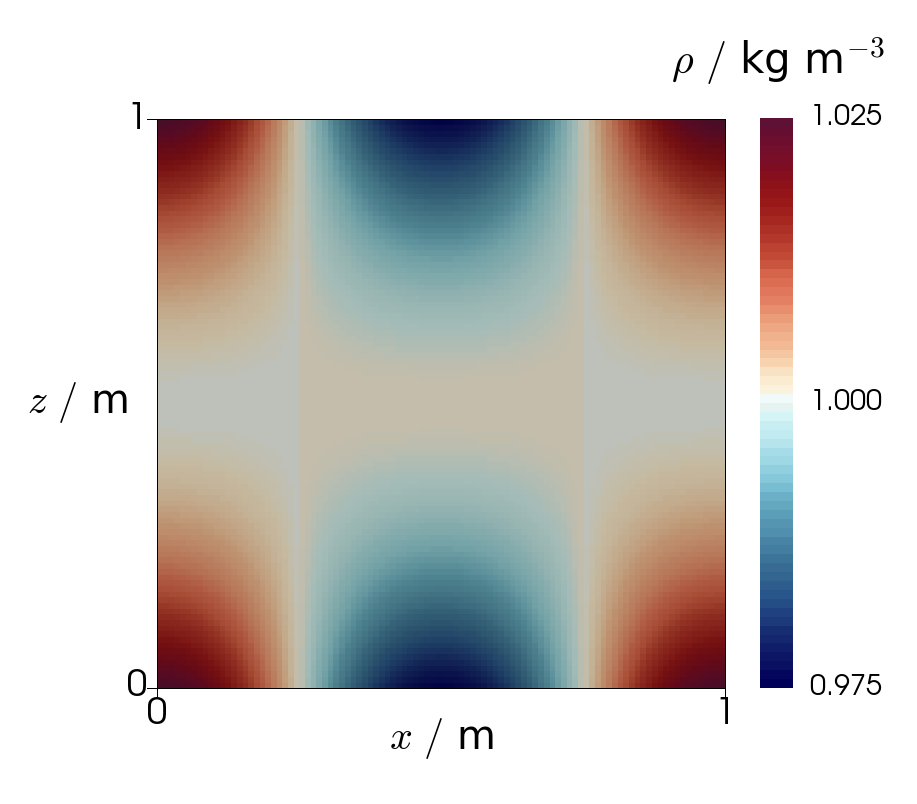}
\end{subfigure}
~~
\begin{subfigure}{.48\textwidth}
\centering
\includegraphics[width=\textwidth]{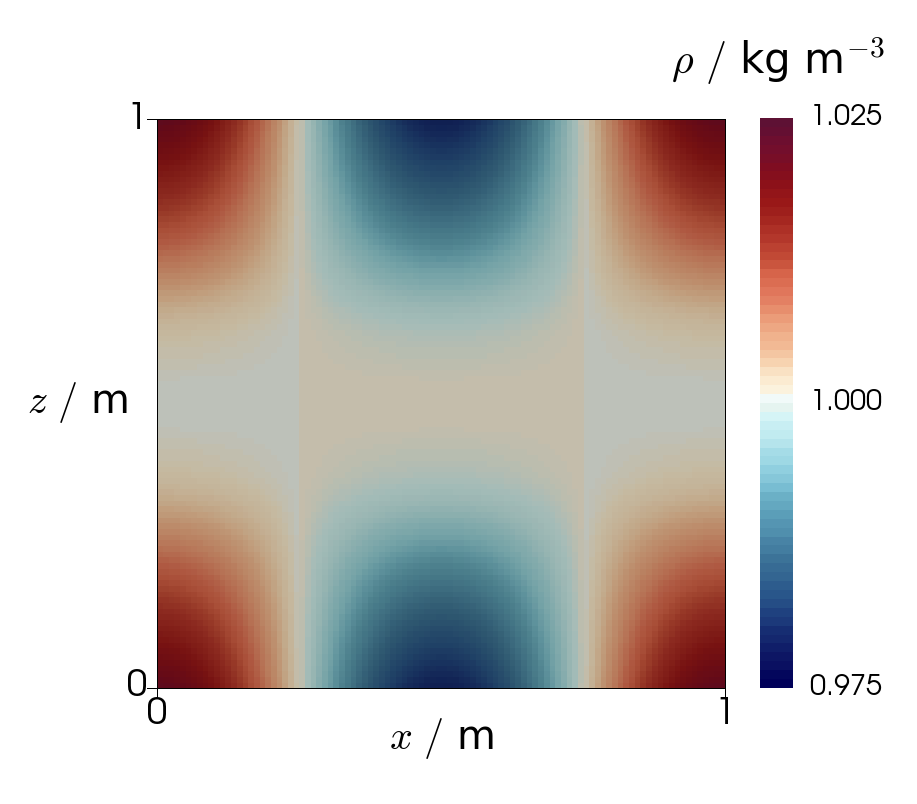}
\end{subfigure}
\caption*{\small{\newB{\textbf{Figure 5}:
Initial (left) and final (right) fields in the $\mathrm{DG}_0\times\mathrm{DG}_0$ space of the boundary 
convergence test of Section \ref{sec:third convergence}, showing the field labelled
$\rho$ in Figure 7 (left) at the lowest resolution ($\Delta x=0.01$ m).}
}}
\end{figure}
The third test was intended to investigate the integrity
of the advection scheme at the boundaries of the domain.
We use the following reversible flow, which squashes the
advected material into the boundary before recovering it:
\begin{equation}
\bm{v}(\bm{x}, t) = \left \lbrace
\begin{matrix}
(1,  - \sin(2\pi z)) & t < 0.5 \\
(1, \sin(2\pi z)) & t \geq 0.5
\end{matrix} \right. .
\end{equation}
The initial condition was
\begin{equation}
q = 1 + \frac{1}{10}\left(z - \frac{1}{2}\right)^2
\cos(2\pi x).
\end{equation}
To see the effect of the extra recovery performed at the 
boundary described in Section \ref{sec:recovery operator}, this extra recovery
was turned off for the variables labelled with an asterisk.
The results in Figure 7 (left) demonstrate that without doing extra
recovery at the boundaries, the whole recovery process does
not have second order numerical accuracy.
\newB{Initial and final fields for the density in the lowest resolution run ($\Delta x = 0.01$ m)
are displayed in Figure 5.}
\subsubsection{Spherical Convergence Test} \label{sec:last convergence}
The final convergence test was performed
on the surface of a sphere.
In this case we used a cubed sphere mesh of a sphere of radius
100 m.
The advecting velocity field used was $\bm{v}=U\sin\lambda$,
for latitude $\lambda$ and $U=\pi / 10$ m s$^{-1}$,
which gave a constant zonal rotation rate about the sphere.
We took time steps of $\Delta t=0.5$ s
up to a total time of $2000$ s
so that the initial profile should be equal to the 
`true' profile.
The initial profile that we used was very similar to that
used in the first test case of \cite{williamson1992standard}:
\begin{equation}
q = \left\lbrace
\begin{matrix}
\frac{1}{2}\left[1 + \cos\left(\frac{\pi r}{R}\right) \right],
& r<R, \\
0, & \mathrm{otherwise}, 
\end{matrix} \right.
\end{equation}
where $R=100/3$ m and for latitude $\lambda$ and longitude 
$\varphi$ with $\lambda_0=0$ and $\varphi_0=-\pi/2$, 
and where $r$ is now given by
\begin{equation}
r = 100 \ \mathrm{m} \times \cos^{-1}\left[\sin\lambda_0\sin\lambda
+\cos\lambda_c\cos\lambda\cos(\varphi-\varphi_0)\right].
\end{equation}
The errors of this test as a function of resolution are plotted
in Figure 7 (right).
This also appears to show second order accuracy.
We found that at lower resolutions, the errors due to the 
advective scheme were obscured by those from the imperfect 
discretisation of the surface of the sphere.
\newB{The initial and final fields of this test are plotted for the coarsest resolution in Figure 6.}
\begin{figure}[h!]
\centering
\begin{subfigure}{.48\textwidth}
\centering
\includegraphics[width=\textwidth]{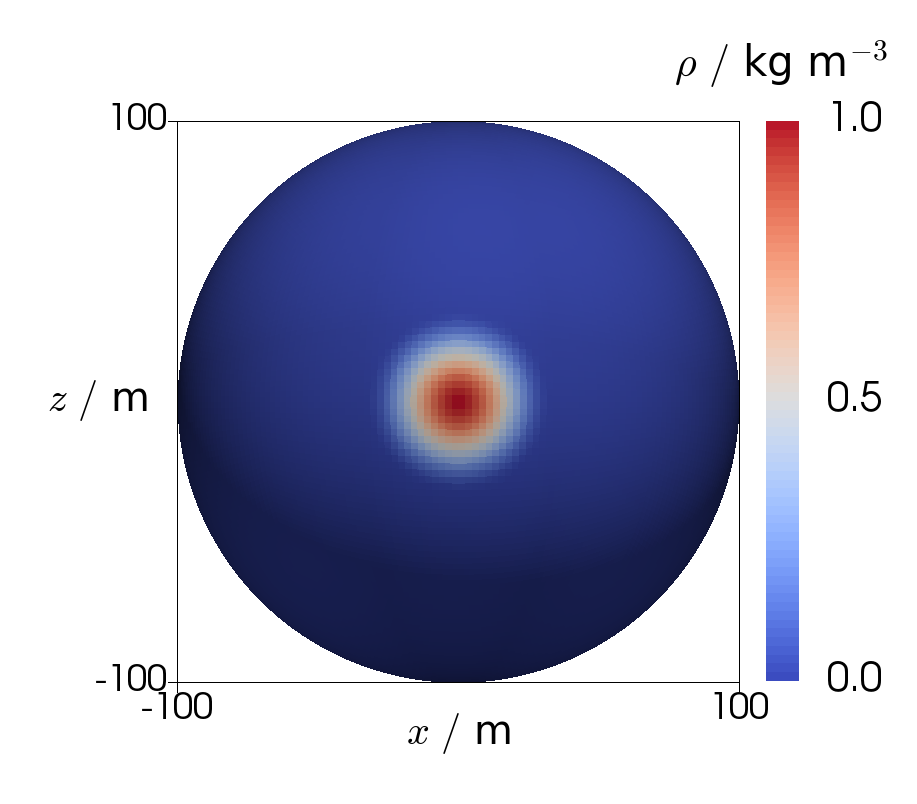}
\end{subfigure}
~~
\begin{subfigure}{.48\textwidth}
\centering
\includegraphics[width=\textwidth]{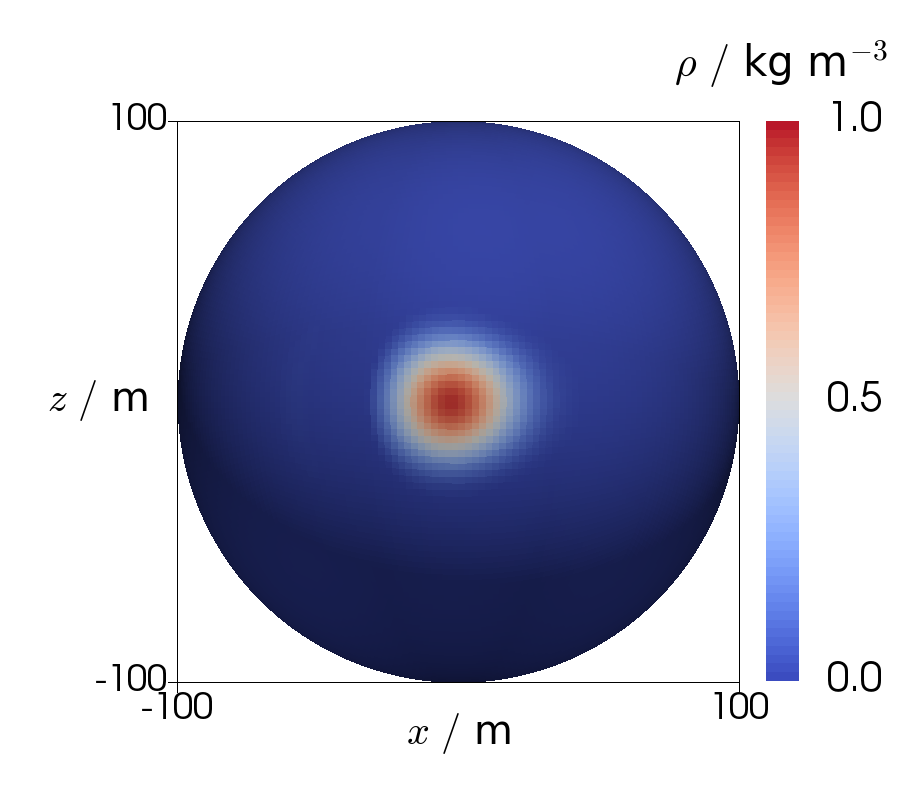}
\end{subfigure}
\caption*{\small{\newB{\textbf{Figure 6}:
The initial (left) and final (right) fields in the $\mathrm{DG}_0\times\mathrm{DG}_0$ space from the spherical 
convergence test of Section \ref{sec:second convergence}, showing the field labelled
$\rho$ in Figure 7 (right) at the lowest resolution (the sixth refinement level of the cubed sphere, with $\Delta x\approx 1.6$ m).}
}}
\end{figure}
\begin{figure}[h!]
\centering
\begin{subfigure}{.48\textwidth}
\centering
\includegraphics[width=\textwidth]{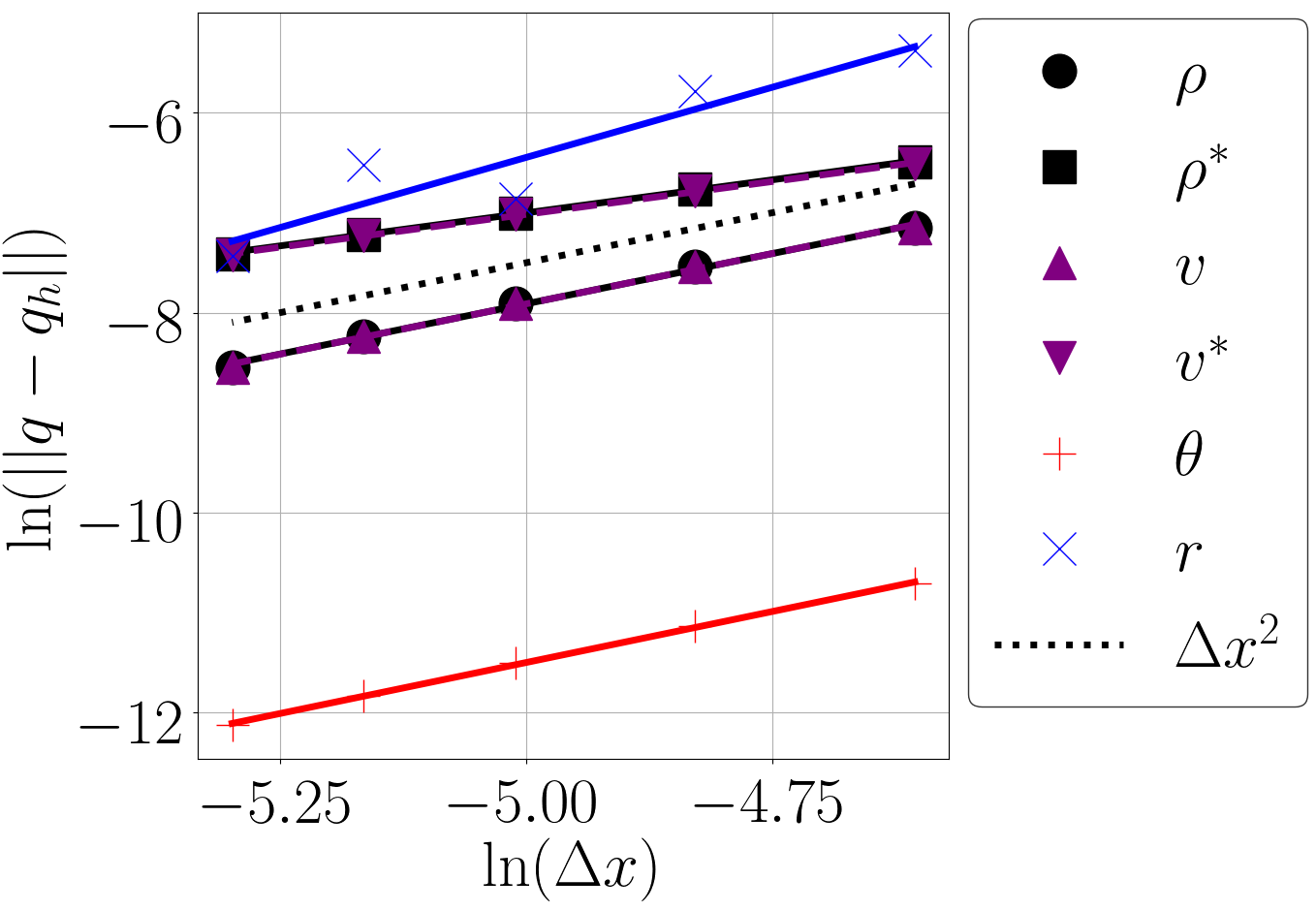}
\end{subfigure}
~~
\begin{subfigure}{.48\textwidth}
\centering
\includegraphics[width=\textwidth]{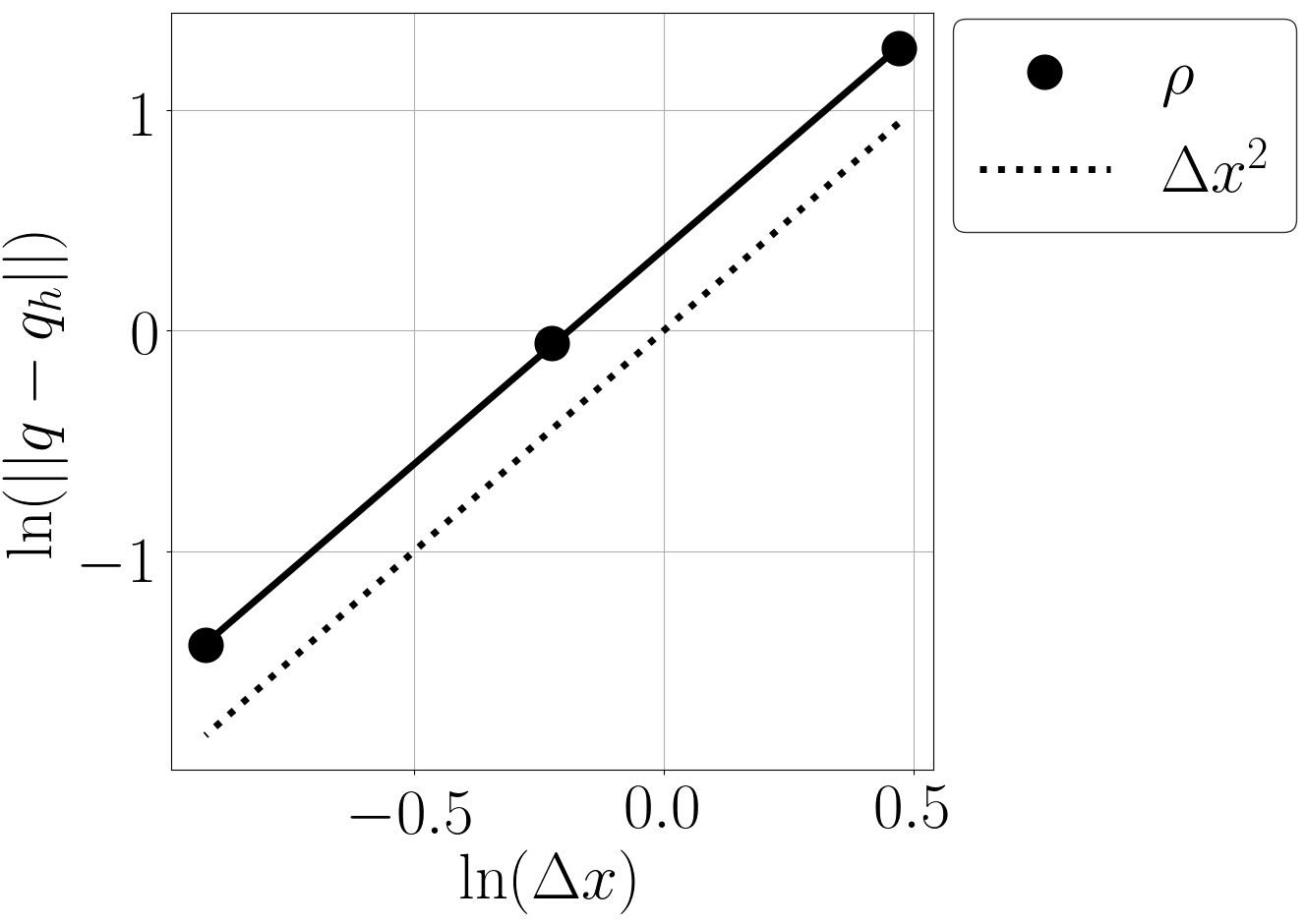}
\end{subfigure}
\caption*{\small{\textbf{Figure 7}:
More results from convergence tests for the recovered space 
scheme, plotting, as a function of grid spacing $\Delta x$, the error in 
an advected solution $q$ against the true solution $q_h$.
(Left) A test demonstrating the need for the extra recovery
at the boundaries,  by comparing the scheme with and without
this extra recovery process.
The schemes without extra recovery at the boundaries are denoted
with an asterisk.
They not only display a larger error, but also show lower accuracy.
\newB{As $\theta$ and $r$ are linear in the vertical, they are accurately represented at the boundary by the recovery scheme without performing any additional recovery at the boundary. 
However if rigid walls were present on the side of the domain, these fields would require additional recovery at these boundaries.}
(Right) The test performed on a cubed sphere mesh.
The \newR{slope} here is very close to $2$, again supporting the claim that the advection scheme has second order numerical accuracy.
}}
\end{figure}

\subsection{Stability}
We tested the formulas (\ref{eqn:A_k1}), (\ref{eqn:A_k2})
and (\ref{eqn:A_k3})
by advecting sine and cosine waves for each of the cases defined.
The domain used was a square vertical slice of length 120 m
with grid spacing $\Delta x=1$ m.
The amplification factor for a given wavenumber $k$ and Courant
number $c$ was measured by advecting a sine and cosine wave
of wavenumber $k$ by a constant horizontal velocity $c$
for a single time step of $\Delta t=1$ s.
As before, the domain had periodic boundary conditions on the
vertical walls.
The amplification factor was then found by measuring the
amplitude of the sine and cosine components after the first time
step.
This was done for several values of $c$.\\
\\
The measured values are compared with those from the formula
in Figure 8
which shows agreement for each of the cases considered in
Section \ref{sec: von neumann}.

\begin{figure}[h!]
\centering
\includegraphics[width=\textwidth]{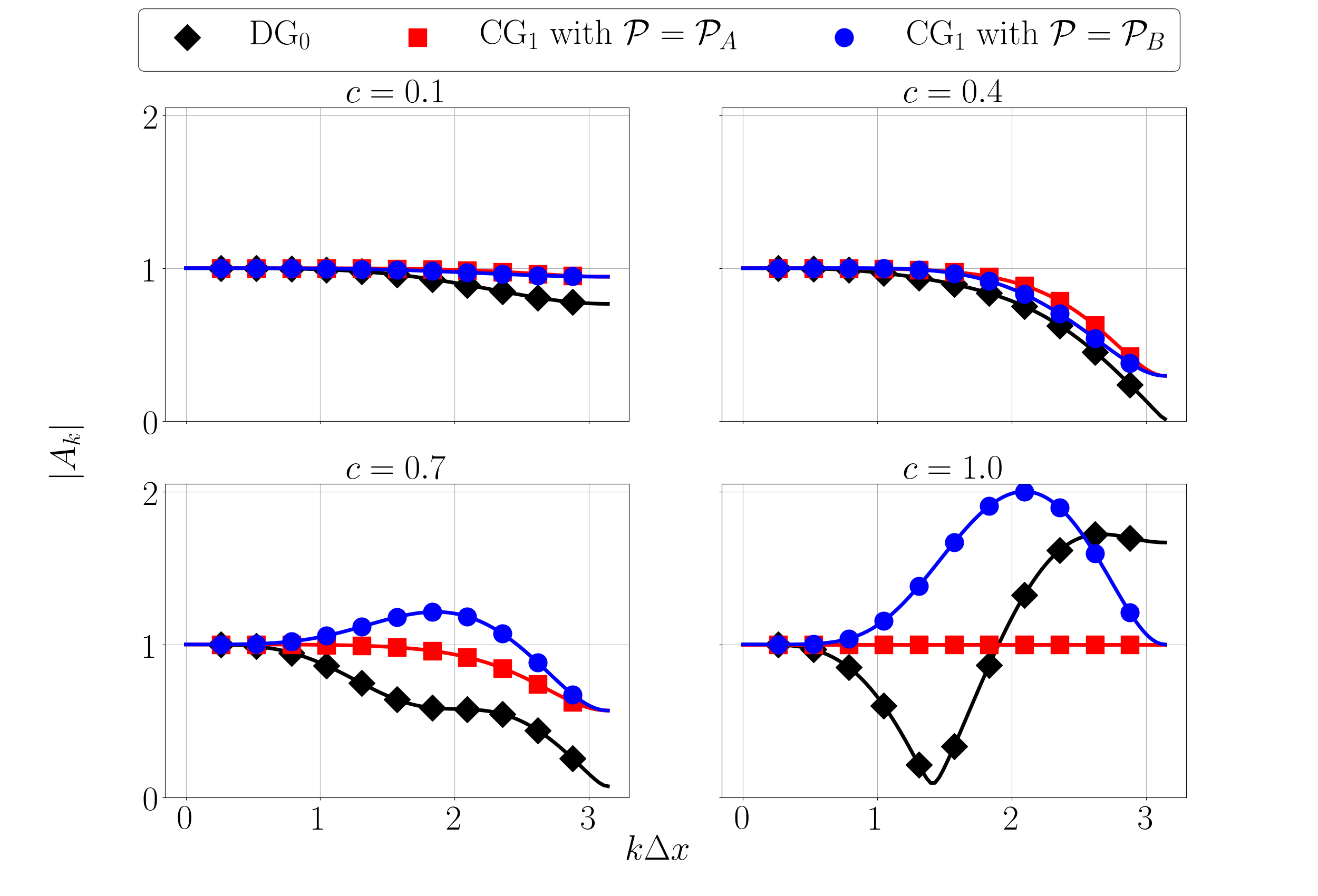}
\caption*{\small{\textbf{Figure 8}:
The results from testing the validity of the expressions
(\ref{eqn:A_k1}), (\ref{eqn:A_k2}) and (\ref{eqn:A_k3})
for the amplification factors in the 1D advection cases
presented in Section \ref{sec: von neumann}.
The markers denote measurements of the amplification factor
by advecting sine and cosine wave profiles,
whilst the lines plot the expressions derived in
Section \ref{sec: von neumann}.
All plots show agreement between the expressions and the measured
amplification factors. \newR{These results also agree with the critical Courant numbers found in Table 2, including for $c=0.4$ in which the values are all below unity for the $\mathrm{CG}_1$ case with $\mathcal{P}_A$, but with some values above unity for the $\mathcal{P}_B$ case.} 
}}
\end{figure}

\subsection{Limiting}
\begin{figure}[h!]
\centering
\begin{subfigure}{0.48\textwidth}
\centering
\includegraphics[width=\textwidth]{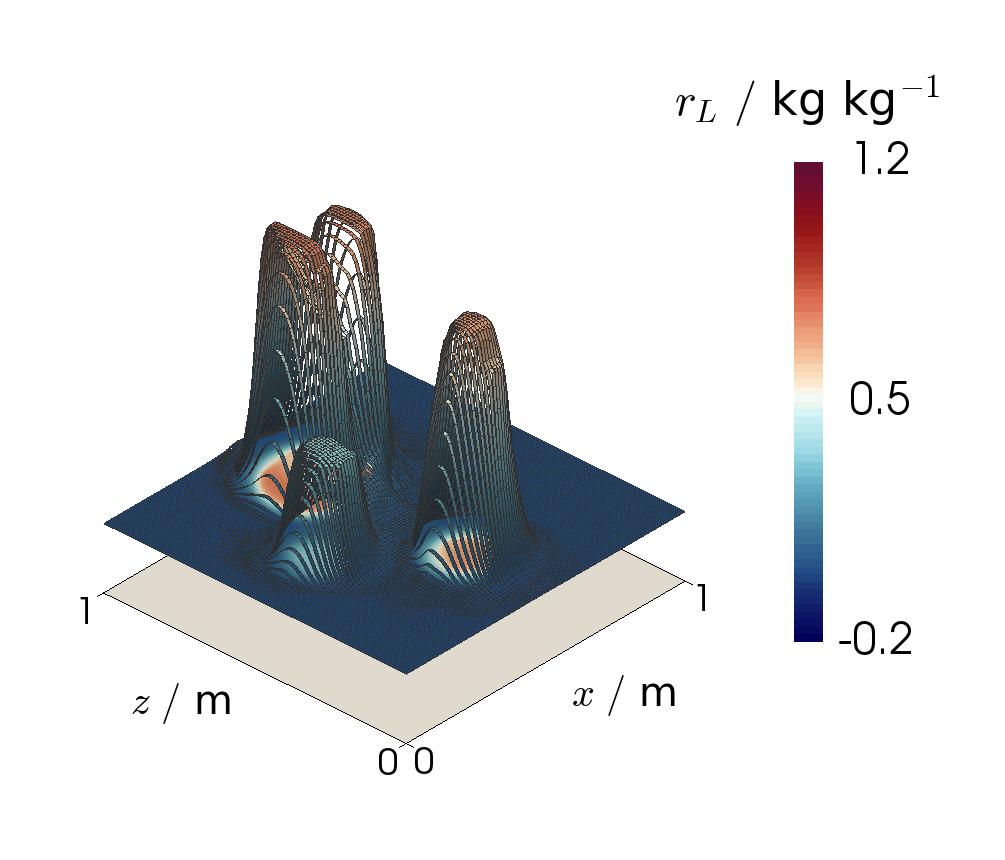}
\end{subfigure}
\begin{subfigure}{0.48\textwidth}
\centering
\includegraphics[width=\textwidth]{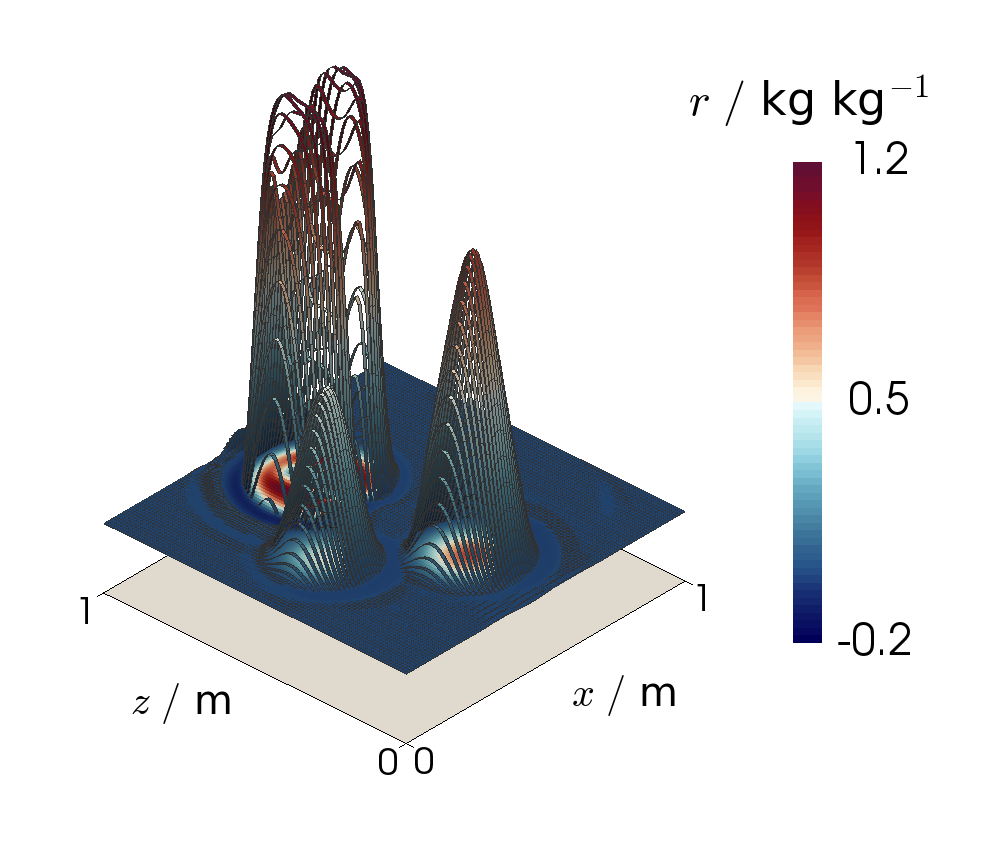}
\end{subfigure}
\begin{subfigure}{0.48\textwidth}
\centering
\includegraphics[width=\textwidth]{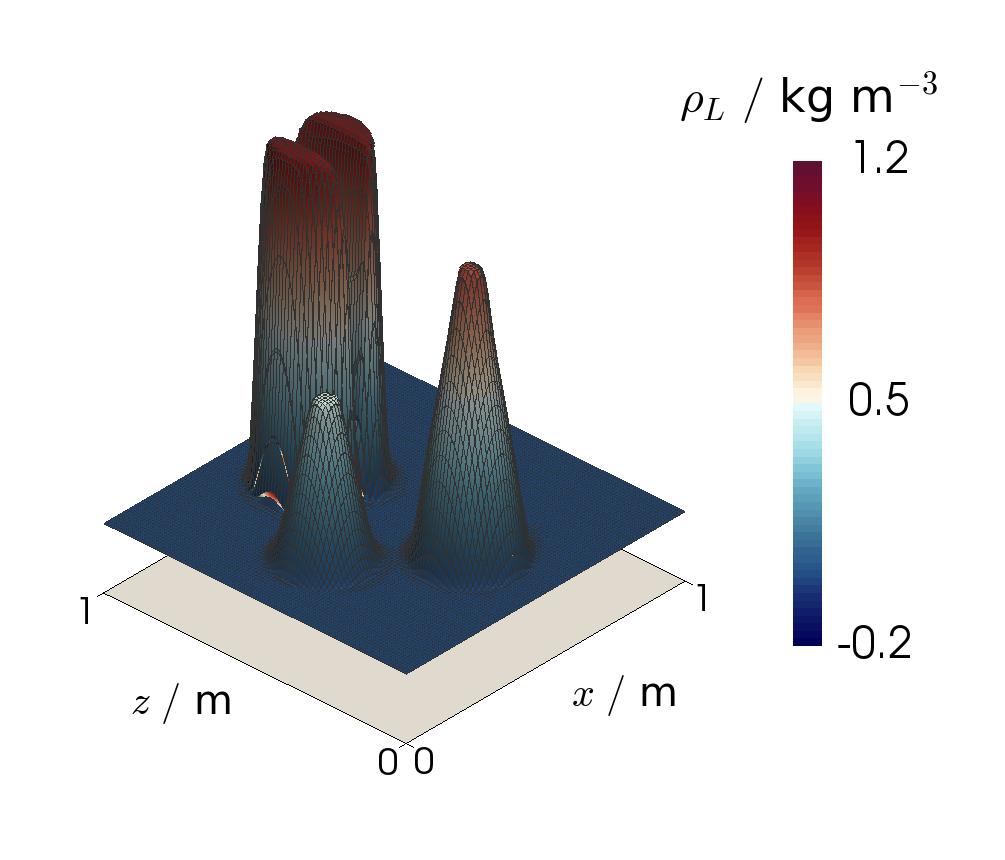}
\end{subfigure}
\begin{subfigure}{0.48\textwidth}
\centering
\includegraphics[width=\textwidth]{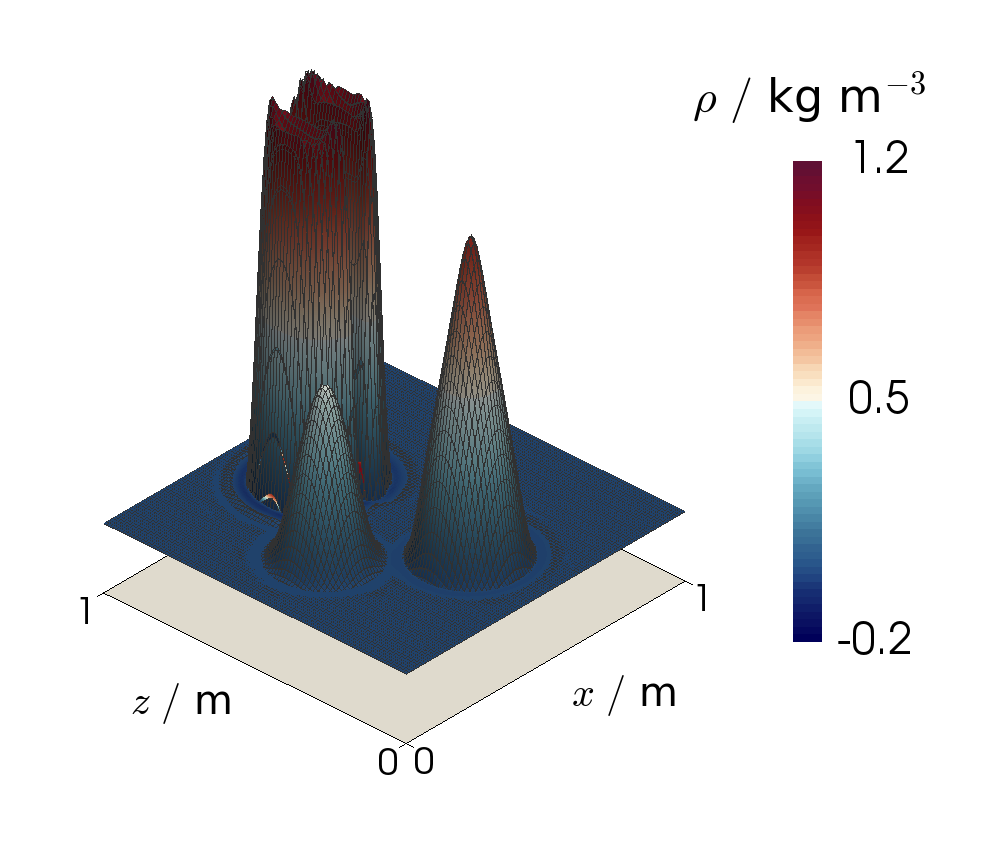}
\end{subfigure}
\caption*{\small{\textbf{Figure 9}:
The resulting fields from one revolution of the solid-body
rotation case of \cite{leveque1996high}.
\newB{(Top left) A hypothetical moisture field in the
$\mathrm{DG}_0\times\mathrm{CG}_1$ space, advected using the
limited `recovered' space scheme, compared with (top right) the same field advected using the non-limited scheme.
Although overshoots and undershoots are prevented, conservation of mass is compromised.
(Bottom left) A density field in the $\mathrm{DG}_1\times\mathrm{DG}_1$
space, using the same advection operator $\mathcal{A}$ as
in the `recovered' scheme and limited by the vertex-based limiter
of \cite{kuzmin2010vertex}, with (bottom right) the same solution but without a limiter applied.
This shows the effectiveness of the limiter in preventing overshoots and undershoots.}
}}
\end{figure}
\noindent
The efficacy of the limiting scheme was tested by using the
LeVeque slotted-cylinder, hump, cone set-up originally
defined in \cite{leveque1996high}
and used in both
\cite{kuzmin2010vertex} and \cite{cotter2016embedded}.
The advected field was initialised lying with this condition,
lying in the $\mathrm{DG}_0\times\mathrm{CG}_1$ space to
mimic moisture variables, before a
solid-body rotation was completed.
This was performed for the bounded case of the scheme defined
in Section \ref{sec:scheme}, using the projection operator
$\mathcal{P}_B$ and the vertex-based limiter of 
\cite{kuzmin2010vertex} for the advection.
The resulting field is shown in Figure 9
where it is also compared to the rotation of a field in
the $\mathrm{DG}_1\times\mathrm{DG}_1$ space, using the
same limited advection scheme, but without the `recovered' parts
of the scheme.
The field does indeed remain bounded, suggesting that the limiter has worked well.

\section{Compressible Euler Model}\label{sec: compressible model}
We have used this advection scheme in a numerical model
of the compressible Euler equations.
The continuous equations we used are
\begin{align}
\pfrac{\bm{v}}{t}+(\bm{v}\bm{\cdot}\grad{})\bm{v} + c_p \theta
\grad{\Pi} - \bm{g} = \bm{0} , \\
\pfrac{\rho}{t} + \grad{}\bm{\cdot}(\rho \bm{v}) = 0 , \\
\pfrac{\theta}{t} + \bm{v}\bm{\cdot}\grad{\theta} = 0 , \\
\Pi = \left(\frac{R\theta\rho}{p_0}\right)^
{\frac{\kappa}{1-\kappa}}
\end{align}
where $\bm{g} = -g\hat{\bm{k}}$ with 
$g=9.81$ m s$^{-2}$ is the uniform gravitational
acceleration towards the Earth's surface, 
$c_p = 1004.5$ J kg$^{-1}$ K$^{-1}$ is the specific heat capacity 
at constant pressure of a dry ideal gas,
$R=287$ J kg$^{-1}$ K$^{-1}$ is the specific gas constant of a
dry ideal gas, $\kappa = R/c_p$
and $\Pi$ is the Exner pressure, found at reference pressure
$p_0=1000$ hPa.\\
\\
The general strategy to solve these equations is based
upon that used in the UK Met Office Endgame model,
and is very similar to that described in
\cite{yamazaki2017vertical}, which discretised the Boussinesq
equations. \\
\\
The overall structure of our model
can be described by the performance of various operations
to the state variables, which we denote together by
$\bm{\chi} = (\bm{v}, \rho, \theta)$.
In a time step, the first stage is to apply a `forcing' 
$\mathcal{F}(\bm{\chi})$ to $\bm{\chi}$, the details of which are explained
below.
At this point the algorithm enters an outer iterative loop, in which
an advecting velocity $\bar{\bm{u}}$ is determined, and an 
advection step is performed by an advection operator 
$\mathcal{V}_{\bar{\bm{u}}}$.
There is then an inner loop, in which the forcing is reapplied
and a residual is calculated between the newly forced state
and the best estimate for the state at the next time step.
The state is corrected by solving a linear problem 
(here we denote the linear operator by $\mathcal{S}$)
for the residual.
The scheme is balanced between being explicit and implicit
by the off-centering parameter $\alpha$, which we will take to be 1/2.
This process is summarised by the following pseudocode:
\begin{enumerate}[leftmargin=*]
\item FORCING:  $\bm{\chi}^*=\bm{\chi}^n+(1-\alpha)\Delta t\mathcal{F}(\bm{\chi}^n)$
\item SET:   $\bm{\chi}_p^{n+1}=\bm{\chi}^n$
\item OUTER:
	\begin{enumerate}
	\renewcommand{\labelitemi}{}
	\item UPDATE:   $\bar{\bm{u}}=\tfrac{1}{2}(\bm{v}_p^{n+1}+\bm{v}^n)$
	\item ADVECT: $\bm{\chi}_p=\mathcal{V}_{\bar{\bm{u}}}(\bm{\chi}^*)$
	\item INNER:
	\begin{enumerate}
	\item FIND RESIDUAL: $\bm{\chi}_\mathrm{rhs}= 
	\bm{\chi}_p+\alpha\Delta t \mathcal{F}(\bm{\chi}_p^{n+1})
	-\bm{\chi}_p^{n+1}	$
	\item SOLVE: $\mathcal{S}(\Delta \bm{\chi}) = 
	\bm{\chi}_\mathrm{rhs}$ for $\Delta\bm{\chi}$
	\item INCREMENT: $\bm{\chi}_p^{n+1}=\bm{\chi}_p^{n+1}+\Delta\bm{\chi}$
	\end{enumerate}
	\end{enumerate}
\item ADVANCE TIME STEP: $\bm{\chi}^n = \bm{\chi}^{n+1}_p$
\end{enumerate}
In our case, the forcing operator acts only upon the velocity.
It is the solution $\bm{v}_\mathrm{trial}$ to the following
problem:
\begin{equation}
\int_\Omega \bm{\psi}\bm{\cdot}\bm{v}_\mathrm{trial} \dx{x}
= \int_\Omega c_p\bm{\nabla}\bm{\cdot}(\theta\bm{\psi}) \dx{x}
- \int_\Gamma c_p \llbracket \theta \bm{\psi}\bm{\cdot}
\hat{\bm{n}} \rrbracket \langle \Pi \rangle \dx{S}
-\int_\Omega g \bm{\psi}\bm{\cdot}\hat{\bm{k}}\dx{x}, \ \ \ 
\forall \bm{\psi}\in V_{\bm{v}},
\end{equation}
where $\Omega$ is the domain, $\Gamma$ is the set of all interior facets, the angled brackets $\langle\cdot\rangle$ denote
the average value on either side of a facet and 
$V_{\bm{v}}$ is the function space 
in which the velocity field lies.\\
\\
The advection operators use the scheme defined in Section 
\ref{sec:scheme} with $\mathcal{A}$ as the simple upwinding
and three-step Runge-Kutta method described in 
Section \ref{sec: von neumann}.
The spaces used in the recovered scheme are listed in 
Section \ref{sec: example spaces}.
\newB{We do not use any limiting strategy for $\theta$, and use the projection operator $\mathcal{P}_A$ in all schemes.} \\
\\
Finally, the strategy for the linear solve step is to first
analytically eliminate $\theta$.
The resulting problem for $\bm{u}$ and $\rho$ defines
the operator $\mathcal{S}$, which we solve using a Schur complement
preconditioner.
Then $\theta$ is reconstructed from the result.
\subsection{Rising Bubble Test Case}
\begin{figure}[h!]
\centering
\begin{subfigure}{0.48\textwidth}
\centering
\includegraphics[width=\textwidth]{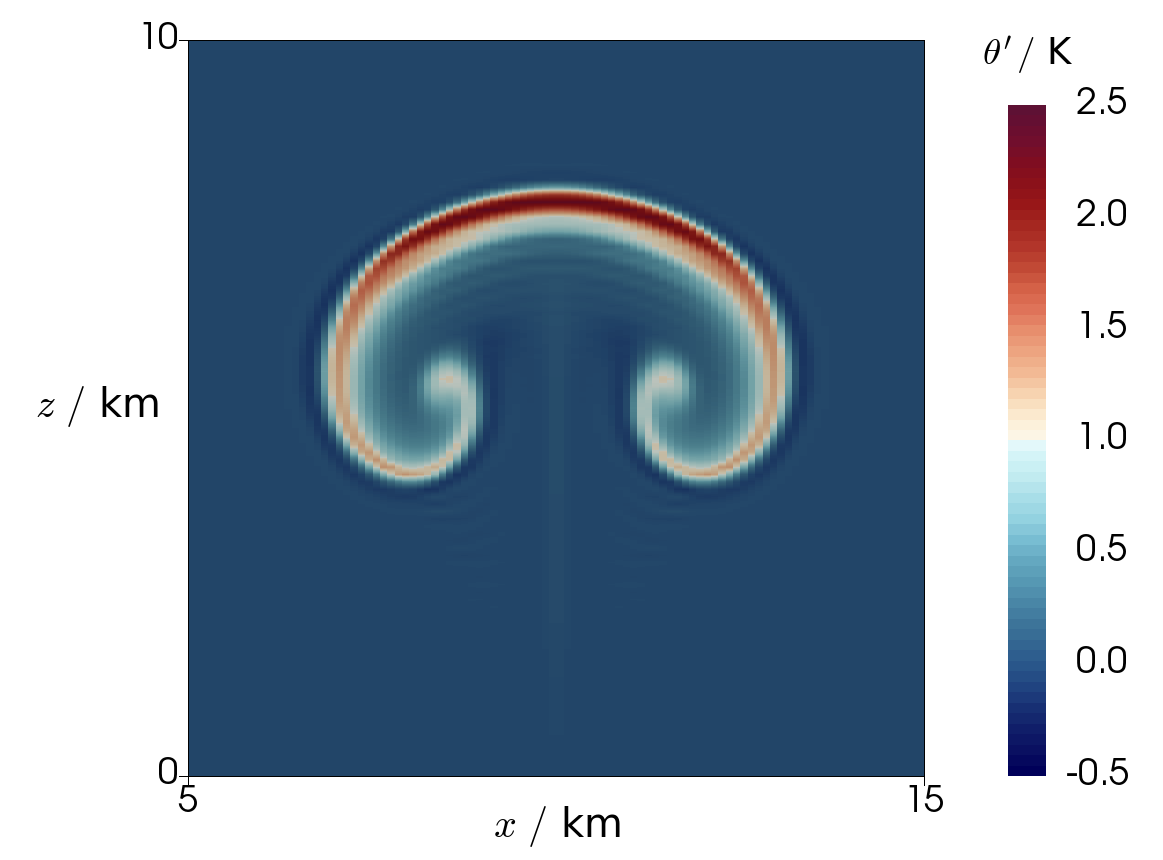}

\end{subfigure}
~~
\begin{subfigure}{0.48\textwidth}
\centering
\includegraphics[width=\textwidth]{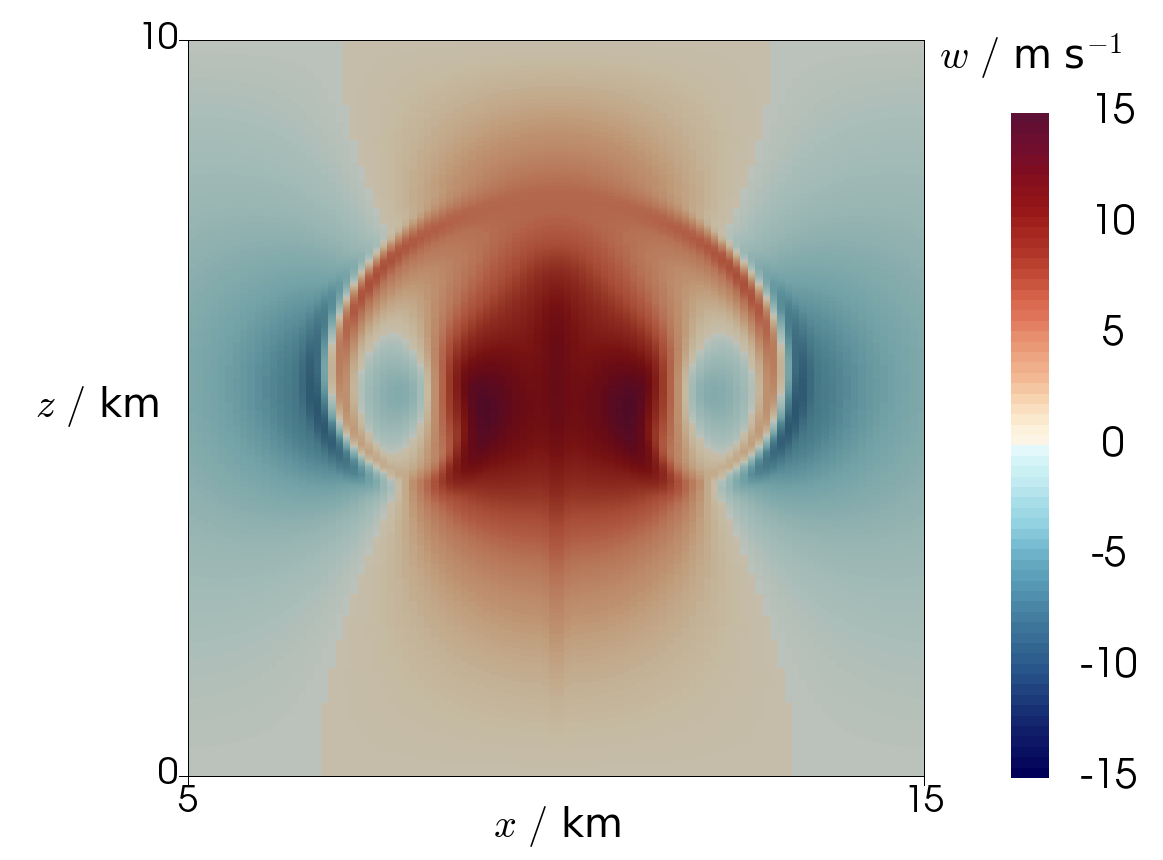}
\end{subfigure}
\caption*{\small{\textbf{Figure 10}:
Fields at $t=1000$ s from a run at resolution $\Delta x=100$ m of 
the dry bubble case from \cite{bryan2002benchmark}, 
representing a rising thermal.
We used the lowest-order family of spaces
with the `recovered' space scheme as the advection method.
(Left) the perturbation $\theta'$
to the background constant state of 300 K. (Right) The vertical
velocity field.
}}
\end{figure} \noindent
Here we show some results of using the recovered space 
advection scheme within a full model of the compressible Euler
equations in the context of a vertical slice model.
The example test case that we use is the dry bubble test of
\cite{bryan2002benchmark}.
The initial state is $\theta_b=300$ K and zero velocity everywhere,
while $\rho$ is determined via from solving for hydrostatic
balance using the procedure described in
\cite{natale2016compatible}.
The following perturbation to $\theta$ was then applied:
\begin{equation}
\theta' = \left\lbrace
\begin{matrix}
2 \cos^2\left(\pi r/(2 r_c)\right) \ \mathrm{K}, & r<r_c, \\
0, & r \geq r_c, 
\end{matrix} \right.
\end{equation}
so that $\theta = \theta_b + \theta'$, with 
\begin{equation}
r = \sqrt{(x-x_c)^2+(z-z_c)^2},
\end{equation}
where $x_c$ = 10 km, $z_c=r_c=$ 2 km.
In the model used in \cite{bryan2002benchmark}, 
the Exner pressure is a prognostic variable, rather than the
density $\rho$.
To ensure that our initial pressure is unchanged by the 
perturbation, we found the initial density state by
solving for $\rho_\mathrm{trial}$:
\begin{equation}
\int_\Omega \phi \rho_\mathrm{trial} \dx{x} =
\int_\Omega \phi \rho_b\theta_b / \theta \dx{x}, \ \ \ 
\forall \phi\in V_\rho.
\end{equation}
where $V_\rho$ is the function space that $\rho$ lives in, and
$\theta_b$ and $\rho_b$ are the hydrostatically balanced
background states.
The domain used had a width of 20 km and a height of 10 km.
The top and bottom boundaries had rigid lid boundary conditions
($\bm{v}\bm\cdot \hat{\bm{n}}=0$) whilst there were periodic
boundary conditions on the left and right sides.
The perturbed potential temperature field at the final time
$t=1000$ s is shown in Figure 10 for a simulation with 
grid spacing $\Delta x=100$ m and time steps of $\Delta t$ = 1 s.
\newR{Although there is good agreement between the fields shown in Figure 10 and those presented in \cite{bryan2002benchmark}, there are some visible errors in the solution: chiefly some ringing artifacts
surrounding the mushroom-shaped perturbation that develops in the $\theta$ field.
In the authors' experience, such numerical simulations of rising or sinking thermals can often exhibit instability at the leading edge of the moving bubble, capturing an inherent physical instability.
In this case, the underlying equations may have multiple solutions and the numerical solution may not converge to one as the resolution is refined.}
\section{Summary and Outlook}
\label{sec:outlook}
We have presented a new `recovered' advection scheme for use in numerical
weather prediction models.
This scheme is a form of the embedded DG advection described in
\cite{cotter2016embedded}, but in which higher-degree spaces
are recovered via averaging operators, as described in
\cite{georgoulis2018recovered}.
It is intended for use with the compatible finite element
set-up laid out in \cite{cotter2012mixed}, and in particular
can be used with the zeroth-degree set of spaces.
With these spaces, the scheme has second-order numerical 
accuracy.
We have also presented a bounded version of this scheme,
which can be used for moisture variables to preserve monotonicity
or to prevent negative values.
Stability properties of the scheme have also been provided.
\newB{Future work should extend this scheme to cover vector functions that lie in curved domains.
This will require careful averaging of vectors that lie in different tangent spaces.}

\section*{Acknowledgements}
TMB was supported by the EPSRC Mathematics of Planet Earth Centre for Doctoral Training at Imperial College London and the University of Reading.
CJC was supported by EPSRC grant EP/L000407/1, while JS was supported by the EPSRC EP/L000407/1 and NERC NE/R008795/1 grants.
The authors would like to thank the developers of the 
Firedrake software which was used extensively for this work, and in particular Thomas H. Gibson who wrote the original piece of code for the recovery operator.
\newP{The authors are also grateful to the two anonymous reviewers whose constructive feedback was useful in the revision of this paper.
Any mistakes, however, belong to the authors.}

\bibliography{recovered_space}

\begin{thebibliography}{10}

\bibitem{staniforth2012horizontal}
A.~Staniforth and J.~Thuburn, ``Horizontal grids for global weather and climate
  prediction models: a review,'' {\em Quarterly Journal of the Royal
  Meteorological Society}, vol.~138, no.~662, pp.~1--26, 2012.

\bibitem{cotter2012mixed}
C.~J. Cotter and J.~Shipton, ``Mixed finite elements for numerical weather
  prediction,'' {\em Journal of Computational Physics}, vol.~231, no.~21,
  pp.~7076--7091, 2012.

\bibitem{staniforth2013analysis}
A.~Staniforth, T.~Melvin, and C.~Cotter, ``Analysis of a mixed finite-element
  pair proposed for an atmospheric dynamical core,'' {\em Quarterly Journal of
  the Royal Meteorological Society}, vol.~139, no.~674, pp.~1239--1254, 2013.

\bibitem{cotter2014finite}
C.~J. Cotter and J.~Thuburn, ``A finite element exterior calculus framework for
  the rotating shallow-water equations,'' {\em Journal of Computational
  Physics}, vol.~257, pp.~1506--1526, 2014.

\bibitem{mcrae2014energy}
A.~T. McRae and C.~J. Cotter, ``Energy-and enstrophy-conserving schemes for the
  shallow-water equations, based on mimetic finite elements,'' {\em Quarterly
  Journal of the Royal Meteorological Society}, vol.~140, no.~684,
  pp.~2223--2234, 2014.

\bibitem{natale2016compatible}
A.~Natale, J.~Shipton, and C.~J. Cotter, ``Compatible finite element spaces for
  geophysical fluid dynamics,'' {\em Dynamics and Statistics of the Climate
  System}, vol.~1, no.~1, 2016.

\bibitem{shipton2018higher}
J.~Shipton, T.~Gibson, and C.~Cotter, ``Higher-order compatible finite element
  schemes for the nonlinear rotating shallow water equations on the sphere,''
  {\em Journal of Computational Physics}, vol.~375, pp.~1121--1137, 2018.

\bibitem{yamazaki2017vertical}
H.~Yamazaki, J.~Shipton, M.~J. Cullen, L.~Mitchell, and C.~J. Cotter,
  ``Vertical slice modelling of nonlinear {E}ady waves using a compatible
  finite element method,'' {\em Journal of Computational Physics}, vol.~343,
  pp.~130--149, 2017.

\bibitem{cotter2016embedded}
C.~J. Cotter and D.~Kuzmin, ``Embedded discontinuous {G}alerkin transport
  schemes with localised limiters,'' {\em Journal of Computational Physics},
  vol.~311, pp.~363--373, 2016.

\bibitem{arnold2014periodic}
D.~N. Arnold and A.~Logg, ``Periodic table of the finite elements,'' {\em SIAM
  News}, vol.~47, no.~9, p.~212, 2014.

\bibitem{mcrae2016automated}
A.~T. McRae, G.-T. Bercea, L.~Mitchell, D.~A. Ham, and C.~J. Cotter,
  ``Automated generation and symbolic manipulation of tensor product finite
  elements,'' {\em SIAM Journal on Scientific Computing}, vol.~38, no.~5,
  pp.~S25--S47, 2016.

\bibitem{georgoulis2018recovered}
E.~H. Georgoulis and T.~Pryer, ``Recovered finite element methods,'' {\em
  Computer Methods in Applied Mechanics and Engineering}, vol.~332,
  pp.~303--324, 2018.

\bibitem{titarev2002ader}
V.~A. Titarev and E.~F. Toro, ``Ader: Arbitrary high order godunov approach,''
  {\em Journal of Scientific Computing}, vol.~17, no.~1-4, pp.~609--618, 2002.

\bibitem{van2005discontinuous}
B.~Van~Leer and S.~Nomura, ``Discontinuous galerkin for diffusion,'' in {\em
  17th AIAA Computational Fluid Dynamics Conference}, p.~5108, 2005.

\bibitem{karakashian2007convergence}
O.~A. Karakashian and F.~Pascal, ``Convergence of adaptive discontinuous
  {G}alerkin approximations of second-order elliptic problems,'' {\em SIAM
  Journal on Numerical Analysis}, vol.~45, no.~2, pp.~641--665, 2007.

\bibitem{kuzmin2010vertex}
D.~Kuzmin, ``A vertex-based hierarchical slope limiter for p-adaptive
  discontinuous {G}alerkin methods,'' {\em Journal of computational and applied
  mathematics}, vol.~233, no.~12, pp.~3077--3085, 2010.

\bibitem{shu1988efficient}
C.-W. Shu and S.~Osher, ``Efficient implementation of essentially
  non-oscillatory shock-capturing schemes,'' {\em Journal of Computational
  physics}, vol.~77, no.~2, pp.~439--471, 1988.

\bibitem{cockburn2001runge}
B.~Cockburn and C.-W. Shu, ``Runge--kutta discontinuous galerkin methods for
  convection-dominated problems,'' {\em Journal of scientific computing},
  vol.~16, no.~3, pp.~173--261, 2001.

\bibitem{rathgeber2017firedrake}
F.~Rathgeber, D.~A. Ham, L.~Mitchell, M.~Lange, F.~Luporini, A.~T. McRae, G.-T.
  Bercea, G.~R. Markall, and P.~H. Kelly, ``Firedrake: automating the finite
  element method by composing abstractions,'' {\em ACM Transactions on
  Mathematical Software (TOMS)}, vol.~43, no.~3, p.~24, 2017.

\bibitem{bercea2016structure}
G.-T. Bercea, A.~T. McRae, D.~A. Ham, L.~Mitchell, F.~Rathgeber, L.~Nardi,
  F.~Luporini, and P.~H. Kelly, ``A structure-exploiting numbering algorithm
  for finite elements on extruded meshes, and its performance evaluation in
  firedrake,'' {\em arXiv preprint arXiv:1604.05937}, 2016.

\bibitem{homolya2016parallel}
M.~Homolya and D.~A. Ham, ``A parallel edge orientation algorithm for
  quadrilateral meshes,'' {\em SIAM Journal on Scientific Computing}, vol.~38,
  no.~5, pp.~S48--S61, 2016.

\bibitem{williamson1992standard}
D.~L. Williamson, J.~B. Drake, J.~J. Hack, R.~Jakob, and P.~N. Swarztrauber,
  ``A standard test set for numerical approximations to the shallow water
  equations in spherical geometry,'' {\em Journal of Computational Physics},
  vol.~102, no.~1, pp.~211--224, 1992.

\bibitem{leveque1996high}
R.~J. LeVeque, ``High-resolution conservative algorithms for advection in
  incompressible flow,'' {\em SIAM Journal on Numerical Analysis}, vol.~33,
  no.~2, pp.~627--665, 1996.

\bibitem{bryan2002benchmark}
G.~H. Bryan and J.~M. Fritsch, ``A benchmark simulation for moist
  nonhydrostatic numerical models,'' {\em Monthly Weather Review}, vol.~130,
  no.~12, pp.~2917--2928, 2002.

\end{thebibliography}
\bibliographystyle{ieeetr}

\end{document}